\documentclass{amsart}
\usepackage{float}
\usepackage{tikz}
\usepackage{placeins}
\usepackage{amsmath,amssymb,lineno,amsthm,fullpage,parskip,graphicx,MnSymbol}
\usepackage{verbatim}

\usetikzlibrary{positioning}

\newtheorem{theorem}{Theorem}[section]
\newtheorem{lemma}[theorem]{Lemma}
\newtheorem{corollary}[theorem]{Corollary}

\newcommand{\colourable}{properly quasi-transitively colourable}
\newcommand{\uniquelycolourable}{uniquely quasi-transitively colourable}
\title{An Analogue of Quasi-Transitivity for Edge-Coloured Graphs}

\begin{document}

\maketitle

\begin{center}
	\large{Christopher Duffy\footnote{christopher.duffy@usask.ca, Research supported by the Natural Science and Engineering Research Council of Canada}, Todd Mullen}
\end{center}

\begin{center}
Department of Mathematics and Statistics, University of Saskatchewan, CANADA.\\
\end{center}

\begin{abstract}
We extend the notion of quasi-transitive orientations of graphs to 2-edge-coloured graphs.
By relating quasi-transitive $2$-edge-colourings to an equivalence relation on the edge set of a graph, we classify those graphs that admit a quasi-transitive $2$-edge-colouring. 
As a contrast to Ghouil\'{a}-Houri's classification of quasi-transitively orientable graphs as comparability graphs, we find quasi-transitively $2$-edge-colourable graphs do not admit a forbiddden subgraph characterization.
Restricting the problem to comparability graphs, we show that the family of uniquely quasi-transitively orientable comparability graphs is exactly the family of comparabilty graphs that admit no quasi-transitive $2$-edge-colouring.
\end{abstract}

\section{Introduction and Preliminaries}
Using the definition of graph colouring based on homomorphism, one may define a notion of proper vertex colouring for oriented graphs and $2$-edge-coloured graphs that respectively takes into account the orientation of the arcs and the colours of the edges.
In surveying the literature on \emph{oriented colourings} and \emph{$2$-edge-coloured colourings}, one finds seemingly endless instances of similar techniques and results.
For example, in both cases, one obtains an upper bound of $80$ for the respective chromatic number of oriented and $2$-edge-coloured planar graphs by constructing a universal target for homomorphisms of orientations/$2$-edge-colourings of planar graphs \cite{AM98,RASO94}.
For further details of oriented colouring and $2$-edge-coloured colourings see  \cite{SO15} and \cite{MOPRS10}.

A key feature of oriented colourings is that vertices at the end of a directed $2$-path (i.e., \emph{a $2$-dipath}) must receive different colours.
So one may bound the oriented chromatic number of an oriented graph $\overrightarrow{G}$ from below by computing the \emph{$2$-dipath chromatic number}, the chromatic number of $\overrightarrow{G}^2$, where $\overrightarrow{G}^2$ is the undirected graph formed from $\overrightarrow{G}$ by first adding an edge between vertices at the end of an induced $2$-dipath and then removing the orientation from the arcs of $\overrightarrow{G}$ \cite{MAYO12,KMY09}.
From this, one observes that if $\overrightarrow{G}$ has no induced $2$-dipaths, then the $2$-dipath chromatic number of $\overrightarrow{G}$ is equal to the chromatic number of the simple graph underyling $\overrightarrow{G}$.
Oriented graphs with no induced $2$-dipaths are studied in the literature under the name \emph{quasi-transitive digraphs} \cite{B95}.

In \cite{B95}, a recursive method is used to characterize the graphs that can have their edges oriented into a quasi-transitive digraph.
In \cite{H12}, a related term is defined, \emph{$k$-quasi-transitive digraphs}, in which vertices must be adjacent if there exists a directed path of length $k$ connecting them. Much of the research on quasi-transitive digraphs and $k$-quasi-transitive digraphs have focused on \emph{strong} digraphs. A \emph{strong} digraph is a graph in which, for every pair of vertices, there exists a directed path containing those two vertices. A characterization of strong $3$-quasi-transitive digraphs is given in \cite{H10}.

We will opt for the name \emph{quasi-transitive orientations} instead of \emph{quasi-transitive digraphs} to relate more closely to the following two terms.
A graph $G$ is \emph{quasi-transitively orientable} if there exists a quasi-transitive orientation for which $G$ is the underlying graph.
A graph $G$ is \emph{uniquely quasi-transitively orientable} if there exist exactly two quasi-transitive orientations for which $G$ is the underlying graph.
We use the word ``unique" to describe the case in which there are two possible orientations because one can be created by reversing the direction of every arc in the other.

Quasi-transitive orientations arise only as certain orientations of comparability graphs, which are graphs formed from making adjacent related elements of a partial order \cite{GH62}. We include this theorem now for ease of reference.

\begin{theorem} \cite{GH62} \label{thm:ghouila}
	A graph $G$ is quasi-transitively orientable if and only if $G$ is a comparability graph.
\end{theorem}

Unsurprisingly, one may follow a similar line of thought for proper colourings of $2$-edge-coloured graphs.
In doing so, one finds that vertices at the ends of a $2$-path in which one edge is red and the other is blue (i.e., an alternating $2$-path) must receive different colours.
To date, however, no work has been done in the study of $2$-edge-coloured graphs with no induced alternating $2$-path.
In this work we undertake the first such study of these objects and find a full classification of those graphs that admit a $2$-edge-colouring so as to have no induced alternating $2$-path.
For this end, we define the following.

A \emph{quasi-transitive $2$-edge-colouring} of a graph $G$ is a mapping $c:E(G) \to \{R,B\}$ so that for all pairs $xy,yz \in E(G)$, with $c(xy)\neq c(yz)$, we have $xz \in E(G)$.
A quasi-transitive $2$-edge-colouring is \emph{trivial} when $c(e) = c(f)$ for all $e,f \in E(G)$.
In other words, the edge colouring is \emph{monochromatic}.
As trivial quasi-transitive $2$-edge-colourings exist for all graphs, all graphs are quasi-transitively colourable. 
A graph $G$ is \emph{\colourable} if there exists a nontrivial quasi-transitive $2$-edge-colouring of $G$. 
A graph $G$ is \emph{\uniquelycolourable} if there exist exactly 2 nontrivial quasi-transitive $2$-edge-colourings of $G$. 
We use the term unique for graphs with exactly two colourings because the existence of a single nontrivial quasi-transitive $2$-edge-colouring implies the existence of a second, created by interchanging all of the colours. 
Examples of quasi-transitive $2$-edge-colourings are shown in Figure~\ref{fig:QT2EC}.

\begin{figure}[H]
	\centering	
	\begin{tikzpicture}[-,-=stealth', auto,node distance=1.5cm,
		thick,scale=0.8, main node/.style={scale=0.8,circle,draw,font=\sffamily\Large\bfseries}]
		
		\node[main node] (1) 					    {};			
		\node[main node] (2)  [right = 2cm of 1]        {};
		\node[main node] (3)  [below right = 2cm and 1cm of 2]        {};  
		\node[main node] (4)  [below left = 2cm and 1cm of 3]	      {};				
		\node[main node] (5)  [left = 2cm of 4]        {};
		\node[main node] (6)  [above left = 2cm and 1cm of 5]        {};	
		
		\node[main node] (7)  [right = 4cm of 2]					    {};			
		\node[main node] (8)  [right = 2cm of 7]        {}; 
		\node[main node] (9)  [below right = 2cm and 1cm of 8]        {};
		\node[main node] (10)  [below left = 2cm and 1cm of 9]        {};	
		
		\node[main node] (11)  [left = 2cm of 10]        {};  
		\node[main node] (12)  [above left = 2cm and 1cm of 11]        {};

		
		\draw[dotted]
		(1) -- (2)
		(2) -- (3)
		(2) -- (6)
		(5) -- (4)
		(1) -- (4)
		(2) -- (5)
		(3) -- (4)
		(4) -- (6)
		
		(7) -- (8)
		(8) -- (9)
		(9) -- (10)
		(10) -- (11)
		(11) -- (12)
		(7) -- (12)
		(7) -- (11)
		(8) -- (12)
		(9) -- (11)
		(10) -- (12);
		

		\draw[]
		(1) -- (5)
		(3) -- (5)
		(1) -- (6)
		(5) -- (6);
		
		
	\end{tikzpicture}
	\caption{Quasi-transitive $2$-edge-colourings shown on two graphs with $R$ edges shown as dotted lines and $B$ edges shown as full lines (note that the second graph only has trivial quasi-transitive $2$-edge-colourings)}
	\label{fig:QT2EC}	
\end{figure}
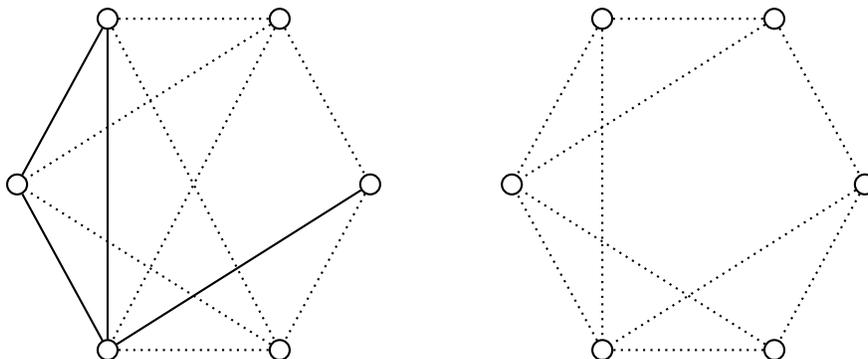	
	
Given a set of edges $E$, we denote by $V(E)$ the set of endpoints of edges in $E$. 
Also, we denote by $G[E]$ the graph with edge set $E$ and vertex set $V(E)$. 
For other graph theoretic terminology not defined herein, we refer the reader to \cite{bondy}.

Our work proceeds as follows.
In Section \ref{sec:SECHF1F2} we introduce an equivalence relation on the set of edges of a graph that partitions the edges into subsets of edges that must be assigned the same colour in any quasi-transitive $2$-edge-colouring.
With the aim of classifying those graphs that are \colourable, we study the structure of the subgraphs induced by the equivalence classes.
Our work yields a full classification of \colourable\ graphs and \uniquelycolourable\ graphs.
In doing so, we provide an infinite family of \uniquelycolourable\ graphs.

In Section \ref{sec:3equivClasses} we provide a characterization of those graphs for which the equivalence relation yields exactly three equivalence classes.
In doing so, we provide examples of infinite families of such graphs.

In Section \ref{sec:QTOrient} we introduce the notion of uniquely quasi-transitively orientable graphs.
By restricting our work to comparability graphs, we find that the equivalence relation introduced for the study of \colourable\ graphs yields insight into those comparability graphs for which there are exactly two quasi-transitive orientations.
In particular, we show that the family of uniquely quasi-transitively orientable graphs is exactly the family of comparability graphs that admit only the trivial quasi-transitive $2$-edge-colouring.

\section{Edge Partitions and Quasi-Transitive $2$-Edge-Colourings}\label{sec:SECHF1F2}

In this section, we show the respective sets of red and blue edges in a quasi-transitive $2$-edge-colouring arise as unions of equivalence classes under an equivalence relation on the edges of a graph. 
From this observation, we show the family of \colourable\ graphs are exactly those graphs with at least two equivalence classes with respect to this equivalence relation.
We begin with a number of results leading up to the definition of this equivalence relation.

Let $G$ be a graph and let $e \in E(G)$.
Let $\mathcal{S}_{e}$ be the set comprising all subsets $S$ of $E(G)$ with the following properties:
\begin{enumerate}
\item $e \in S$; and 
\item for all $d\notin S$, there is no induced copy of $P_3$ in $G$ comprising $d$ and an edge of $S$.
\end{enumerate}

	\begin{theorem}\label{thm:intersection}
		The set $\mathcal{S}_e$ is closed with respect to intersection.
	\end{theorem}
	
	\begin{proof}
		Let $G$ be a graph and let $e \in E(G)$.
		Toward a contradiction, suppose  $S$ and $T$ are two sets in $\mathcal{S}_e$ and  $S \cap T$ is not in $\mathcal{S}_e$.
		Since $e$ is in $S \cap T$ and $S \cap T$ is not in $\mathcal{S}_e$, there must exist some induced copy of $P_3$ containing an edge $d$ in $S \cap T$ and an edge $f$ in $E(G) \setminus (S \cap T)$.
		Therefore, any set in $\mathcal{S}$ that contains every edge in $S \cap T$ must contain $f$ as well.
		This is a contradiction, as $f$ is in at most one of $S$ and $T$.
		Therefore, the intersection of any two sets in $\mathcal{S}_e$ is a set in $\mathcal{S}_e$.
	\end{proof}

	\begin{corollary} \label{lem:fewestunique}
		The element of $\mathcal{S}_{e}$ with the fewest number of edges is unique.
	\end{corollary}	
	 
	 \begin{proof}
	 	Let $S$ and $T$ be elements of $\mathcal{S}_e$.
	 	Toward a contradiction, suppose that no set in $\mathcal{S}_e$ contains fewer edges than either $S$ or $T$.
	 	The set $S \cap T$ contains fewer edges than either $S$ or $T$.
	 	By Theorem~\ref{thm:intersection}, $S \cap T$ is in $\mathcal{S}_e$.
	 	This is a contradiction.
	 	Therefore, the element of $\mathcal{S}_{e}$ with the fewest number of edges is unique.
	 \end{proof}
	 
	 	Denote by $S_e$ the unique set of least order in $\mathcal{S}_e$. 
 
	\begin{corollary} \label{cor:connected}
		The graph $G[S_e]$ is connected.
	\end{corollary}
			
	\begin{proof}
		Toward a contradiction, suppose $G[S_e]$ is disconnected. 
		Let $S_1$ be the largest set of edges in $S_e$ that both contains $e$ and is such that $G[S_1]$ is connected. 
		Since $S_e$ is the unique set of least order in $\mathcal{S}_e$, there must exist some edge $d$ in $S_1$ and some edge $f$ in $S_e \setminus S_1$ such that $d$ and $f$ exist in some induced copy of $P_3$.
		Therefore, $f$ is an edge in $S_e \setminus S_1$ that is incident with an edge in $S_1$.
		This, however, is a contradiction as $S_1$ is the largest set of edges in $S_e$ that both contain $e$ and is such that $G[S_1]$ is connected. 
		Thus $G[S_e]$ is connected.
	\end{proof}

With Theorem~\ref{thm:equalsets}, we establish that these sets of least order form a partition on the edges.	

	\begin{lemma}\label{lem:ab}
		For every pair of edges $uv, xy \in E(G)$, if $uv \in S_{xy}$, then $S_{uv} \subseteq S_{xy}$.
	\end{lemma}
	
	\begin{proof}
		Let $uv,xy \in E(G)$ and let $uv \in S_{xy}$. 
		The intersection, $S = S_{xy} \cap S_{uv}$, contains the edge $uv$. 
		If there exists any edge $s$ in $S$ such that $s$ is in an induced copy of $P_3$ with an edge $d$ in $E(G) \setminus S$, then $d$ is in both $S_{xy}$ and $S_{uv}$. 
		This however, cannot be the case since $d$ would then be in the intersection $S$. 
		Thus $S$ is the smallest set of edges that both contains $uv$ and is such that for all edges $e$ in $E(G) \setminus S$, the edge $e$ does not exist in an induced copy of $P_3$ with any edge in $S$. 
		Therefore, $S = S_{uv}$ and $S_{uv} \subseteq S_{xy}$.
	\end{proof}
	
	 \begin{theorem}\label{thm:equalsets}
	 	For every pair of edges $xy$ and $uv$ in $E(G)$, we have $S_{xy} = S_{uv}$ if and only if $xy \in S_{uv}$.
	 \end{theorem}
	
	\begin{proof}
		One direction of this statement is clear because $xy$ is, by definition, in $S_{xy}$. Thus if $S_{xy} = S_{uv}$, then $xy \in S_{uv}$.
		
		Consider now  $xy \in S_{uv}$ with $xy \neq uv$. 
		Toward a contradiction, suppose  $S_{uv} \neq S_{xy}$. 
		Let $S = S_{uv} \cap S_{xy}$. 
		The set $S$ contains at least one edge, $xy$.
		
		\emph{Case 1:} $S = S_{xy}$\\
		No edge in $S_{uv} \setminus S_{xy}$ exists in an induced copy of $P_3$ with an edge in $S_{xy}$. 
		Thus, by Lemma~\ref{lem:ab}, for all edges $ab$ in $S_{xy}$, the set $S_{ab}$ does not contain any edges in $S_{uv} \setminus S_{xy}$. 
		This implies that $uv \in S_{uv} \setminus S_{xy}$. 
		However, since $uv \in S_{uv} \setminus S_{xy}$ and no edge in $S_{uv} \setminus S_{xy}$ exists in an induced copy of $P_3$ with an edge in $S_{xy}$, we  conclude that $xy$ is not in $S_{uv}$, which is a contradiction.
		
		\emph{Case 2:} $S \neq S_{xy}$\\ 
		Since $xy \in S$ and $S$ is a proper subset of $S_{xy}$, there must exist some edge $ab$ in $E(G) \setminus S$ such that $ab$ exists in an induced copy of $P_3$ with some edge in $S$. 
		This implies that for all edges $e$ in $E(G)$, if $S$ is a subset of $S_e$, then $S_e$ contains $ab$. 
		However, this implies that $ab$ is in both $S_{xy}$ and $S_{uv}$, but not in $S_{xy} \cap S_{uv}$. 
		This is a contradiction. 
		
		Therefore $S_{uv} = S_{xy}$.
	\end{proof}
	
\begin{corollary}\label{cor:dichromaticargument}
	Let $G$ be a connected graph and let $uv,vw \in E(G)$.
	If $S_{uv} \neq S_{vw}$, then $uw \in E(G)$.
\end{corollary}	
	
By Theorem \ref{thm:equalsets}, for $e,f \in E(G)$ it follows that either $S_e = S_f$ or $S_e \cap S_f = \emptyset$.
Thus there exists a subset of edges $E^\prime(G) \subset E(G)$ so that
\begin{itemize}
	\item $\bigcup\limits_{e \in E^\prime(G)} S_e = E(G)$; and
	\item $S_e \cap S_f = \emptyset$ for all $e,f \in E^\prime(G)$ with $e \neq f$.
\end{itemize}
Using this partition of the edges, we study quasi-transitive $2$-edge-colourings of graphs.

Let $\mathcal{C}_G$ be the relation on $E(G)$ so that $e \sim f$ when $c(e) = c(f)$ for every quasi-transitive $2$-edge-colouring of $G$. The relation $\mathcal{C}_G$ will serve as a partition on the edges, so $\mathcal{C}_G$ is an equivalence relation. Let $[e]_{\mathcal{C}}$ denote the equivalence class of $e$ with respect to this relation.

\begin{theorem}\label{thm:equivalencerelation2}
	Let $G$ be a graph.
	For every $e \in E(G)$, we have $[e]_{\mathcal{C}} =S_{e}$.
\end{theorem}

\begin{proof}
	Let $G$ be a connected graph and let $e \in E(G)$.
	Let $S = [e]_{\mathcal{C}} \cap S_e$.
	Since $e \in [e]_{\mathcal{C}}$ and $e \in S_e$, necessarily $S$ is non-empty.
	Toward a contradiction, suppose $[e]_{\mathcal{C}} \neq S_{e}$.
	Thus there exists an edge in $E(G)$ that is not in $S$, and since $G$ is connected, there exists an edge in $E(G)$ that is both incident with an edge in $S$ and not, itself, in $S$.
	
	Let $d$ be an edge in $E(G) \setminus S$ that is incident with an edge $f$ in $S$.

	Since the edge $d$ is not in $S$, either $d$ is not in $[e]_{\mathcal{C}}$ or $d$ is not in $S_e$.
	We will show that, in both cases, there does not exist an induced copy of $P_3$ comprising $d$ and $f$.
	
	\textit{Case 1:} $d$ is not in $[e]_{\mathcal{C}}$.
	This implies there exists a quasi-transitive $2$-edge-colouring of $G$ in which $c(f)$ and $c(d)$ are not equal. 
	Thus $df$ is not an induced copy of $P_3$ because that would imply every proper quasi-transitive $2$-edge-colouring of $G$ has $c(d)$ equal to $c(f)$. 
	
	\textit{Case 2:} $d$ is not in $S_e$.
	This implies there does not exist an induced copy of $P_3$ containing $d$ and an edge in $S_e$. Thus $df$ is not an induced copy of $P_3$.

	So no edge in $S$ is contained in an induced copy of $P_3$ with the edge $d$.
	Therefore, since $S$ is a subset of $S_e$, and $S$ contains $e$, we have that $S = S_{e}$.
	However in addition, with no edges in $S$ existing in an induced copy of $P_3$ with an edge in $E(G) \setminus S$, there exist quasi-transitive $2$-edge-colourings of $G$ in which every edge in $E(G) \setminus S$ maps to $R$ and every edge in $S$ maps to $B$.
	So no edges in $E(G) \setminus S$ are in $[e]_{\mathcal{C}}$. 
	
	Therefore, since $S$ is a subset of $[e]_{\mathcal{C}}$, we have $S_{e} = S = [e]_{\mathcal{C}}$.

\end{proof}

\begin{corollary}\label{cor:countingColourings}
	If $G$ is a connected graph such that $\mathcal{C}_G$ has $k$ equivalence classes, then there are $2^{k}$ quasi-transitive $2$-edge-colourings of $G$.
\end{corollary}

\begin{proof}
	For every edge $e$ in $E(G)$, in a quasi-transitive $2$-edge-colouring, either $c(e) = R$ or $c(e) = B$.
	Since every edge in an equivalence class must map to the same value, there are $k$ equivalence classes, and every edge must map to one of two values, there exist $2^k$ quasi-transitive $2$-edge-colourings of $G$. 
\end{proof}

With our next theorem, we show that there exist graphs with any positive integral number of equivalence classes. This theorem will rely on some graph theory terminology that we must first define.

A \textit{threshold graph} is a graph constructed from a single vertex by repeatedly performing one of two operations: adding an independent vertex or adding a universal vertex.
A \textit{split graph} is a graph such that the vertices can be partitioned into a complete graph and an independent set.

\begin{theorem}\label{thm:threshold}
	For every integer $k \geq 1$, there exists a graph $G$ such that $\mathcal{C}_G$ has exactly $k$ equivalence classes.
\end{theorem}

\begin{proof}
	Let $G_n$ be a threshold graph with $n$ vertices, where $n$ is an even integer, constructed by alternating the two possible actions, beginning with adding a universal vertex to a copy of $K_1$ (see Figure~\ref{fig:threshold}). $G_n$ is a split graph in which the complete graph and vertex set have the same number of vertices.
	
			\begin{figure}[H]
		\centering
		
		\begin{tikzpicture}[-,-=stealth', auto,node distance=1.5cm,
			thick,scale=0.4, main node/.style={scale=0.8,circle,draw,font=\sffamily\Large\bfseries}]
			
			\node[main node] (1) 					        {$v_1$};
			\node[main node] (2)  [right = 1cm of 1]        {$v_2$}; 
			\node[main node] (3)  [right = 2cm of 2]        {$v_1$};
			\node[main node] (4)  [right = 1cm of 3]        {$v_2$};
			\node[main node] (5)  [above = 1cm of 3]        {$v_3$};
			\node[main node] (6)  [above = 1cm of 4]        {$v_4$};
			\node[main node] (7)  [right = 2cm of 4]        {$v_1$};
			\node[main node] (8)  [right = 1cm of 7]        {$v_2$};			
			\node[main node] (9)  [above = 1cm of 7]        {$v_3$};			
			\node[main node] (10)  [above = 1cm of 8]       {$v_4$};			
			\node[main node] (11)  [above = 1cm of 9]       {$v_5$};	
			\node[main node] (12)  [above = 1cm of 10]      {$v_6$};
			
			\draw[]
			(1) -- (2)
			(3) -- (4)
			(7) -- (8);
			
			\draw[dashed]
			(5) -- (3)
			(5) -- (4)
			(5) -- (6)
			(9) -- (7)
			(9) -- (8)
			(9) -- (10)
			;
			
			\draw[dotted]
			(11) -- (8)
			(11) -- (9)
			(11) -- (10)
			(11) -- (12)
			;
			
			\draw[dotted] (11) edge [out=200,in=160,looseness=1] (7);
			
		\end{tikzpicture}
		
		\caption{Threshold graphs constructed by alternating the two possible actions. Equivalence classes are shown by solid, dashed, and dotted lines.}
		\label{fig:threshold}
		
	\end{figure}
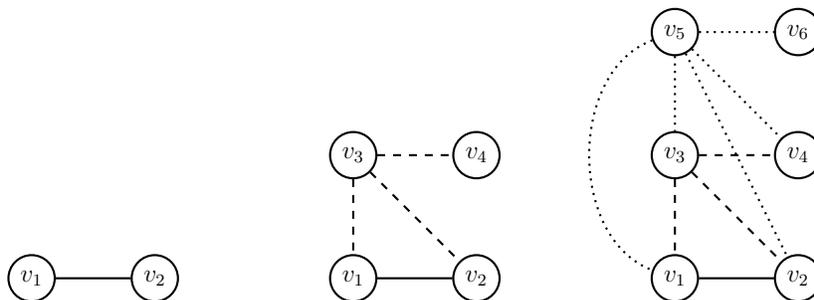

	Let $e$ be an edge in $G_n$. In $G_{n+k}$, for all even integers $k \geq 2$, $S_e$ will be the same set as it is in $G_n$  (assuming the vertices receive the same label in both graphs, as pictured in Figure~\ref{fig:threshold}) because every added vertex is either adjacent to every vertex in $S_e$ or none of the vertices in $S_e$.
	So $G_{n+2}$ has more equivalence classes than $G_n$ for all $n \geq 2$.
	The two vertices, $v_{n+1}$ and $v_{n+2}$, added to $G_n$ to create $G_{n+2}$ are adjacent and this edge, $v_{n+1}v_{n+2}$, induces a copy of $P_3$ with every other edge that is in $G_{n+2}$ and not in $G_n$.
	Therefore, $G_{n+2}$ has exactly one more equivalence class than $G_n$, for all $n \geq 2$.
	Since $G_2$ has only one edge and thus only one equivalence class, the graph $G_n$ has exactly $\frac{n}{2}$ equivalence classes for all even positive integers $n$. 
\end{proof}

To classify those graphs that admit only trivial quasi-transitive $2$-edge-colourings, it suffices to classify those graphs for which $S_e = S_f$ for all $e,f \in E(G)$.
Similarly, to classify those graphs that are uniquely quasi-transitively $2$-edge-colourable, it suffices to classify those graphs for which there exists a pair of edges $e,f \in E(G)$ so that $E(G) = S_e \cup S_f$ and $S_e \neq S_f$.
For these ends, we define the following notation and terminology.

Let $G$ be a graph and let $e \in E(G)$.
An \emph{$S_e$-path} is a path in $G[S_e]$. 
An \emph{$S_e$-$k$-path} is a path of length $k$ in $G[S_e]$. 
For $v \in V(G)$, denote by $N_{G[S_e]}[v]$ the closed neighbourhood of $v$ in $G[S_e]$.

With an eye towards a classification of those graphs with at least $2$-equivalence classes (and thus, also a classification of those graphs with exactly one equivalence class), these next two lemmas show that in every graph that admits a non-trivial quasi-transitive $2$-edge-colouring, there exists an equivalence class $S_e$ for which $|V(S_e)| \neq V(G)$.

	\begin{lemma}\label{lem:ECHF1F2}
	Let $G$ be a graph.
	If $H$ is an induced proper  subgraph of $G$ such that 
	\begin{itemize}
		\item $H$ is connected;
		\item $2 \leq |V(H)| \leq n-1$;  and 
		\item for every vertex $u$ in $V(G) \setminus V(H)$, if $u$ is adjacent to a vertex in $V(H)$, then $u$ is adjacent to every vertex in $V(H)$;
	\end{itemize}  	
	then for all edges $e$ in $E(H)$,  we have $S_e \subseteq E(H)$.
\end{lemma}

\begin{proof}
	Let $G$ be a graph and let $H$ be an induced proper subgraph of $G$ so that 
	\begin{itemize}
		\item $H$ is connected;
		\item $2 \leq |V(H)| \leq n-1$;  and 
		\item for every vertex $u$ in $V(G) \setminus V(H)$, if $u$ is adjacent to a vertex in $V(H)$, then $u$ is adjacent to every vertex in $V(H)$.
	\end{itemize}  		
	Let $e \in E(H)$.
	Toward a contradiction, suppose that some edge $f$ is in $S_e$ and not in $E(H)$.
	This implies that some endpoint, $v_f$, of $f$ is not in $V(H)$.
	Since $S_e$ is defined to be the smallest set in $\mathcal{S}_e$, every edge in $S_e$ must exist in some induced copy of $P_3$ with another edge in $S_e$.
	Also, every edge in $S_e$ must be contained in some path containing $e$ in which every pair of incident edges belong to an induced copy of $P_3$ in $G$.
	Without loss of generality, suppose that $f$ is incident with some edge in $E(H)$ along a path containing $e$ in which every pair of incident edges belong to an induced copy of $P_3$ in $G$.
	So $v_f$ is adjacent to some vertex in $V(H)$, but $v_f$ is not adjacent to every vertex in $V(H)$.
	This is a contradiction.
	Therefore, $S_e \subseteq E(H)$.
\end{proof}

	\begin{lemma}\label{lem:noteveryec}
		Let $G$ be a connected graph and let $e$ and $f$ be edges in $E(G)$.
		If $S_e \neq S_f$, then $V(S_e) \neq V(S_f)$.
	\end{lemma}

	\begin{proof}		
		Let $G$ be a connected graph and let $e$ and $f$ be edges in $E(G)$ such that $S_e \neq S_f$. 
		Toward a contradiction, suppose that $V(S_e) = V(S_f)$. 
		Let $uv \in S_f$.
		The set $N_{G[S_e]}[u]$ is not empty since $u$ is incident with at least one edge in $S_e$.
		Let $G[S_e] - N_{G[S_e]}[u]$ be the induced subgraph of $G[S_e]$ containing the vertices of $V(G[S_e]) \setminus N_{G[S_e]}[u]$.
		Let $B_e(u,v)$ be the set of vertices in the component containing $v$ in $G[S_e] - N_{G[S_e]}[u]$.

		Let $C_e(u,v)$ be the set of all vertices in $V(G)$ that are in neither $N_{G[S_e]}[u]$ nor $B_e(u,v)$.
		The set $C_e(u,v)$ may be empty.
					
		\emph{Case 1: There exists some choice of edge $e$ and adjacent vertices $u$ and $v$ such that $uv \in S_f$, and the set $B_e(u,v)$ has at least two vertices.}\\
		We show every vertex in $V(G) \setminus B_e(u,v)$ that is adjacent to a vertex in $B_e(u,v)$ must be adjacent to every vertex in $B_e(u,v)$. 
		Recall that $V(S_e) = V(S_f)$.
		Since all induced copies of $P_3$ have both or neither of its edges in $S_e$, if there exists a vertex in $V(G) \setminus V(S_e)$, then it is either adjacent to every vertex in $V(S_e)$ or none of the vertices in $V(S_e)$.
		
		Every vertex $b$ in $B_e(u,v)$ exists in some $S_e$-path with endpoints $b$ and $v$. 
		Every vertex in $B_e(u,v)$ that is adjacent to $v$ with an edge in $S_e$ must be adjacent to $u$ with an edge not in $S_e$, by Corollary~\ref{cor:dichromaticargument}. 
		By induction, for some $k$, suppose that every vertex $b$ in $B_e(u,v)$ that exists in an $S_e$-$k$-path with endpoints $b$ and $v$ is adjacent to $u$ with an edge not in $S_e$.
		Let $q$ be a vertex that exists in an $S_e$-$(k+1)$-path with endpoints $q$ and $v$. 
		Since the neighbour of $q$ in this path is adjacent to $u$ with an edge not in $S_e$, by Corollary~\ref{cor:dichromaticargument}, $q$ is adjacent to $u$ with an edge not in $S_e$.
		Therefore, for every vertex $b$ in $B_e(u,v)$, $b$ is adjacent to $u$ with an edge that is not in $S_e$.
		By Corollary~\ref{cor:dichromaticargument}, every vertex in $B_e(u,v)$ is adjacent to every vertex in $N_{G[S_e]}[u]$.
		
		There does not exist any edge in $S_e$ connecting a vertex in $B_e(u,v)$ to a vertex in $C_e(u,v)$ because this would imply that these two vertices exist in the same component of $G[S_e] - N_{G[S_e]}(u)$.
		So for all pairs of adjacent vertices $b \in B_e(u,v)$ and $c \in C_e(u,v)$, the edge $bc$ is not in $S_e$.
		By induction suppose that $c$ is adjacent to $b$ and every vertex $b^\prime$ in $B_e(u,v)$ such that $b^\prime$ exists in an $S_e$-$k$-path with $b$ as an endpoint, for some $k \geq 1$.
		Let $b^{\prime\prime}$ be a vertex such that there exists an $S_e$-$(k+1)$-path with endpoints $b$ and $b^{\prime\prime}$.
		Since the neighbour of $b^{\prime\prime}$ in this path is adjacent to $c$ with an edge not in $S_e$, by Corollary~\ref{cor:dichromaticargument}, the edge $b^{\prime\prime}c$ is in $E(G)$.
		Therefore, every vertex in $V(G) \setminus B_e(u,v)$ that is adjacent to a vertex in $B_e(u,v)$ must be adjacent to every vertex in $B_e(u,v)$. 
		So there does not exist any induced copy of $P_3$ in $G$ such that one edge has both endpoints in $B_e(u,v)$ and the other does not.
	
		Let $x$ be a vertex in $B_e(u,v)$ such that $vx \in S_e$.
		Since $vx$ has both endpoints in $B_e(u,v)$, by Lemma~\ref{lem:ECHF1F2}, the equivalence class $S_{vx}$ is a (perhaps improper) subset of the edges with both endpoints in $B_e(u,v)$.
		Since $N_{G[S_e]}[u]$ is nonempty, there exists some edge in $S_e$ that is not in $B_e(u,v)$.
		However, since $vx \in S_e$, we have that $S_{vx} = S_e$.
		This is a contradiction.
		So $V(S_e) \neq V(S_f)$.		
		
	\emph{Case 2: For all choices of edge, $e$, and adjacent vertices $u$ and $v$ such that $uv \in S_f$, the set $B_e(u,v)$ only contains $v$.}\\		
		Recall that $v$ must have at least one neighbour via an edge in $S_e$ because $V(S_e) = V(S_f)$.
		Let $x$ be a neighbour of $v$ such that $vx$ is in $S_e$.
		Since $x$ is not in $B_e(u,v)$, there does not exist an $S_e$-path connecting $x$ to $v$ that contains no vertices in $N_{S_e}[u]$.
		Since the edge $xv$ is in $S_e$, and $v$ is not in $N_{S_e}[u]$, the vertex $x$ must be in $N_{G[S_e]}[u]$.
			
				\begin{figure}[H]
			\centering

			\begin{tikzpicture}[-,-=stealth', auto,node distance=1.5cm,
				thick,scale=0.4, main node/.style={scale=0.8,circle,draw,font=\sffamily\Large\bfseries}]
				
				\node[main node] (1) 					        {$v$};			
				\node[main node] (3)  [below right = 3cm and 2.5cm of 1]        {$x$}; 
				\node[main node] (4)  [right = 2cm of 1]	       {$u$};			
				\node[main node] (7)  [below right= 1.5cm and 3cm of 4]	       {};			
				\node[main node] (8)  [below right = 1cm and 0.5cm of 7]        {};
				\node[main node] (9)  [below left = 0.7cm and 0.5cm of 7]        {};
				\node[draw=none, fill=none] (100)  [below = 4cm of 1]	       {$B_e(u,v)$};			
				\node[draw=none, fill=none] (101)  [right = 1.5cm of 100]        {$N_{G[S_e]}[u]$};
				\node[draw=none, fill=none] (102)  [right = 1.5cm of 101]        {$C_e(u,v)$};

				\draw
				(-4,-10)--(-4,2)--(3,2)--(3,-10)--(-4,-10)
				(4,-10)--(4,-2)--(11,-2)--(11,-10)--(4,-10)
				(12,-10)--(12,-2)--(19,-2)--(19,-10)--(12,-10);
				
				\draw[]
				;
				\draw[dashed]
				(1) -- (3)
				
				(3) -- (4)
				;
				\draw[dotted]
				(1) -- (4)
				;
				
			\end{tikzpicture}

			\caption{Graph $G$ shown with vertices partitioned into sets  $B_e(u,v)$, $N_{G[S_e]}[u]$, and $C_e(u,v)$. $S_e$ is represented by dashed edges and $S_f$ is represented by dotted edges.}
			\label{fig:AllECs2}
		\end{figure}
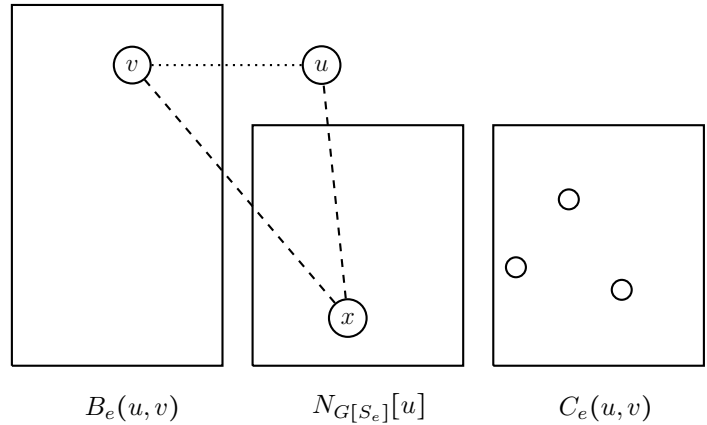
		
		The vertex $v$ is a neighbour of $x$ via an edge that is not in $S_{f}$, and $v$ is a neighbour of $u$ via an edge in $S_{f}$.
		This implies that $u$ is in $N_{G[S_{f}]}[x]$ since otherwise, there would exist an $S_{f}$-path from $u$ to $v$ that contains no vertices in $N_{G[S_{f}]}[x]$, contradicting our supposition that $B_{f}(x,v)$ only contains $v$ (see Figure~\ref{fig:AllECs3}).
		However, the edge $xu$ is in $S_e$, which is distinct from $S_{f}$.
		This is a contradiction. 			
		So $V(S_e) \neq V(S_f)$.

				\begin{figure}[H]
			\centering

			\begin{tikzpicture}[-,-=stealth', auto,node distance=1.5cm,
				thick,scale=0.4, main node/.style={scale=0.8,circle,draw,font=\sffamily\Large\bfseries}]
				
				\node[main node] (1) 					        {$v$};			
				\node[main node] (3)  [below right = 3cm and 1.025cm of 1]        {$u$}; 
				\node[main node] (4)  [right = 2cm of 1]	       {$x$};			
				\node[main node] (7)  [below right= 1.5cm and 3cm of 4]	       {};			
				\node[main node] (8)  [below right = 1cm and 0.5cm of 7]        {};
				\node[main node] (9)  [below left = 0.7cm and 0.5cm of 7]        {};
				\node[draw=none, fill=none] (100)  [below = 4cm of 1]	       {$B_{f}(x,v)$};			
				\node[draw=none, fill=none] (101)  [right = 1.3cm of 100]        {$N_{G[S_{f}]}[x]$};
				\node[draw=none, fill=none] (102)  [right = 1.3cm of 101]        {$C_{f}(x,v)$};

				\draw
				(-4,-10)--(-4,2)--(3,2)--(3,-10)--(-4,-10)
				(4,-10)--(4,-2)--(11,-2)--(11,-10)--(4,-10)
				(12,-10)--(12,-2)--(19,-2)--(19,-10)--(12,-10);
				
				\draw[]
				;
				\draw[dashed]
				
				(3) -- (4)
				(1) -- (4)
				;
				\draw[dotted]
				(1) -- (3)
				;
				
			\end{tikzpicture}

			\caption{Graph $G$ shown with vertices partitioned into sets $B_{f}(x,v)$, $N_{G[S_{f}]}[x]$, and $C_{f}(x,v)$. $S_e$ is represented by dashed edges and $S_f$ is represented by dotted edges. As $u$ must be in both $N_{G[S_{f}]}[x]$ and $B_{f}(x,v)$, a contradiction arises.}
			\label{fig:AllECs3}
		\end{figure}
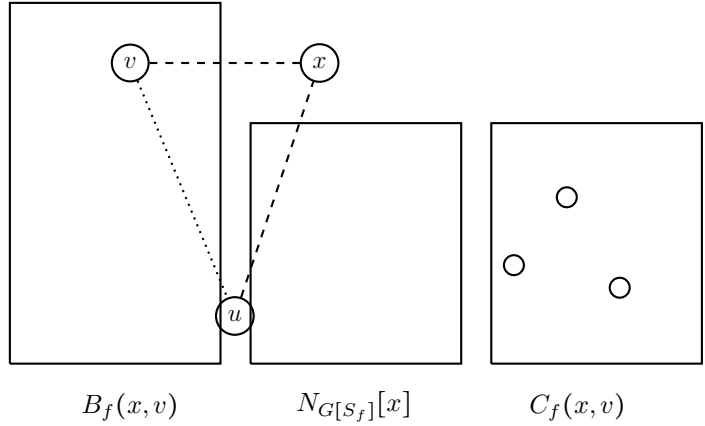
	\end{proof}

\begin{corollary}\label{cor:subset}
	Let $G$ be a connected graph.
	If $\mathcal{C}_G$ has exactly two equivalence classes, $S_e$ and $S_f$, then either $V(S_e) \subset V(S_f)$ or $V(S_f) \subset V(S_e)$.
\end{corollary}

Using Lemma \ref{lem:noteveryec} we provide our main result of this section: a classification of those connected graphs that admit a non-trivial quasi-transitive $2$-edge-colouring.
	
	\begin{theorem} \label{thm:HF1F2}		
		A connected graph $G$ is \colourable\ if and only if there exists an  induced proper subgraph $H$ of $G$ such that 		
		\begin{itemize}
			\item $H$ is connected;
			\item $2 \leq |V(H)| \leq n-1$; and
			\item for every vertex $v$ in $V(G) \setminus V(H)$, if $v$ is adjacent to a vertex in $V(H)$, then $v$ is adjacent to every vertex in $V(H)$.
		\end{itemize}  
	\end{theorem}

	\begin{proof}
		Let $G$ be a connected graph.
		Suppose first that there exists an induced proper subgraph $H$ of $G$ such that 
		\begin{itemize}
	\item $H$ is connected;
	\item $2 \leq |V(H)| \leq n-1$; and
	\item for every vertex $v$ in $V(G) \setminus V(H)$, if $v$ is adjacent to a vertex in $V(H)$, then $v$ is adjacent to every vertex in $V(H)$.
\end{itemize}  		
		From the definition of $H$, it follows that there does not exist an induced copy of $P_3$ in $G$ containing exactly one edge from $E(H)$. 
		Therefore  each of $E(H)$ and $E(G) \setminus E(H)$ arise as the union of some number of equivalence classes under equivalence relation $\mathcal{C}_G$.
		Therefore, the equivalence relation $\mathcal{C}_G$ contains at least two equivalence classes, which means that $G$ is \colourable\ by Theorem~\ref{thm:equivalencerelation2}.
 	
		Suppose now that $G$ is \colourable.
		Thus by Theorem \ref{thm:equivalencerelation2}, the equivalence relation $\mathcal{C}_G$ has at least two equivalence classes.
		By Lemma~\ref{lem:noteveryec}, there exists some equivalence class $S_e$ of $\mathcal{C}_G$ such that $V(S_e) \neq V(G)$.
		Let $H$ be the subgraph induced by $V(S_e)$.
		By Corollary~\ref{cor:connected}, $H$ is connected.
		Since $S_e$ is an equivalence class and there exists at least one vertex in $V(G) \setminus V(S_e)$, it follows that $2 \leq |V(H)| \leq n-1$.

		Let $v$ be a vertex in $V(G) \setminus V(S_e)$ that is adjacent to a vertex $u$ in $V(S_e)$. 
		Since $v$ is in $V(G) \setminus V(S_e)$, the edge $uv$ is not in $S_e$.
		The vertex $u$ has some neighbour $w$ such that $uw$ is in $S_e$.
		By Corollary~\ref{cor:dichromaticargument}, since $v$ is in $V(G) \setminus V(S_e)$, the edge $vw$ exists.
		Use this as a base case and induct on the distance from $u$ along the shortest $S_e$-$k$-path to show that $v$ is adjacent to every vertex in $H$.
		Suppose that for some $k \geq 1$, every vertex $x$ that exists on an $S_e$-$k$-path that also contains $u$ is such that $vx$ is in $E(G)$.
		Let $y$ be a vertex which does not exist on any $S_e$-$k$-path that contains $u$, but does exist on an $S_e$-$(k+1)$-path that contains $u$.
		The vertex $y$ is adjacent to a vertex $z$ that exists on an $S_e$-$k$-path that contains $u$.
		Thus $vz$ is in $E(G)$ and, by Corollary~\ref{cor:dichromaticargument}, $vy$ is in $E(G)$.
		Therefore, $v$ is adjacent to every vertex in $H$.
		So for every vertex $u$ in $V(G) \setminus V(H)$, if $u$ is adjacent to a vertex in $V(H)$, then $u$ is adjacent to every vertex in $V(H)$.
		Thus, our conclusion holds.  
	\end{proof}

	Contrasting this result with the analogous result for quasi-transitively orientable graphs (Theorem~\ref{thm:ghouila}), we see a significant difference in the resulting classification. 
	We return to quasi-transitively orientable graphs in Section \ref{sec:QTOrient}.
	For now, however, we highlight the following difference between quasi-transitively orientable graphs and \colourable\ graphs.

\begin{corollary}\label{cor:noForbidden}
	The family of quasi-transitively colourable graphs admits no forbidden subgraph characterization.
\end{corollary}

\begin{proof}  
	For all graphs $J$, there exists a connected graph $H$ that contains $J$ as a subgraph.
	Let $G$ be the graph created by joining $H$ and $K_1$.
	By Theorem~\ref{thm:HF1F2}, this graph $G$ is \colourable.
	Thus every graph is a subgraph of some quasi-transitively colourable graph.
	Therefore, the family of quasi-transitively colourable graphs admits no forbidden subgraph characterization.
\end{proof}

Theorem \ref{thm:HF1F2} gives a classification of \colourable\ graphs based on the existence of an induced subgraph with particular properties.
So the lack of existence of such a subgraph, and the existence of a unique such subgraph give rise to the following classifications of those graphs which admit only the trivial quasi-transitive $2$-edge-colouring and those which admit a unique quasi-transitive $2$-edge-colouring.

	\begin{corollary}\label{cor:noColouring}
		$G$ is not \colourable\ if and only if for every induced proper connected subgraph $H$ of $G$ such that $E(H) \neq \emptyset$, there exists some vertex $u$ in $V(G) \setminus V(H)$ that is adjacent to at least one vertex in $V(H)$, but not adjacent to all vertices in $V(H)$.
	\end{corollary}
	
		\begin{theorem} \label{thm:UniqueHF1F2}
		A connected graph $G$ with $n$ vertices is \uniquelycolourable\ if and only if there exists exactly one induced proper subgraph $H$ of $G$ such that 	
		\begin{itemize}
			\item $H$ is connected;
			\item $2 \leq |V(H)| \leq n-1$; and 
			\item for every vertex $u$ in $V(G) \setminus V(H)$, if $u$ is adjacent to a vertex in $V(H)$, then $u$ is adjacent to every vertex in $V(H)$.
		\end{itemize} 		
	\end{theorem}
	
	\begin{proof}
		Suppose that a graph $G$ is \uniquelycolourable.
		By Theorem~\ref{thm:equivalencerelation2}, there are exactly two equivalence classes of $\mathcal{C}_G$, call them $S_e$ and $S_f$ for edges $e$ and $f$ in $E(G)$. 
		By Corollary~\ref{cor:subset}, either $V(S_e) \subset V(S_f)$ or $V(S_f) \subset V(S_e)$. 
		By Theorem~\ref{thm:HF1F2}, there exists an induced proper subgraph $H$ of $G$ such that 		
		\begin{itemize}
		\item $H$ is connected;
		\item $2 \leq |V(H)| \leq n-1$; and 
		\item for every vertex $u$ in $V(G) \setminus V(H)$, if $u$ is adjacent to a vertex in $V(H)$, then $u$ is adjacent to every vertex in $V(H)$.
		\end{itemize}   
		
		Since every induced copy of $P_3$ has either both or neither of its edges in $E(H)$, the edges of $H$ make up either an equivalence class or a union of equivalence classes. 
		Since we know that there are only two equivalence classes in $G$, the edges of $H$ are one equivalence class, call it $S_e$, and the rest of the edges comprise the other equivalence class, call it $S_f$. 
		Thus if there is a second induced subgraph $H^\prime$ satisfying our list of requirements, then it must be for the subgraph induced by $V(S_f)$ to make up $H^\prime$. 
		However, by Theorem~\ref{thm:HF1F2}, $V(H^\prime)$ must be a proper subset of $V(G)$, and we know that either $V(S_e)$ or $V(S_f)$ is equal to $V(G)$. 
		Thus, there exists exactly one proper induced subgraph $H$ of $G$ such that
		\begin{itemize}
		\item $H$ is connected;
		\item $2 \leq |V(H)| \leq n-1$; and 
		\item for every vertex $u$ in $V(G) \setminus V(H)$, if $u$ is adjacent to a vertex in $V(H)$, then $u$ is adjacent to every vertex in $V(H)$.
		\end{itemize}
			
		Now suppose there exists a unique proper induced subgraph $H$ of $G$ such that
		\begin{itemize}
			\item $H$ is connected;
			\item $2 \leq |V(H)| \leq n-1$; and 
			\item for every vertex $u$ in $V(G) \setminus V(H)$, if $u$ is adjacent to a vertex in $V(H)$, then $u$ is adjacent to every vertex in $V(H)$.
		\end{itemize}
		By Theorem~\ref{thm:HF1F2}, there exists a nontrivial quasi-transitive 2-edge-colouring of $G$.
		Thus since $E(H)$ is either an equivalence class or some union of equivalence classes, there are at least two equivalence classes of $\mathcal{C}_G$. 
		Toward a contradiction, suppose that $\mathcal{C}_G$ has more than two equivalence classes. This yields two cases: Either $E(H)$ contains multiple equivalence classes of $\mathcal{C}_G$, or $E(G) \setminus E(H)$ contains multiple equivalence classes of $\mathcal{C}_G$.

		\emph{Case 1: $E(H)$ contains multiple equivalence classes of $\mathcal{C}_G$.}\\
		Let $e$ be an edge such that $S_e \subset E(H)$.
		Since $S_e$ is an equivalence class, every induced copy of $P_3$ in $G$ is such that either both or neither of its edges belong to $S_e$.
		The graph $G[S_e]$ is connected by Corollary~\ref{cor:connected}.
		We know $2 \leq |V(G[S_e])| \leq n-1$ since $S_e$ contains at least one edge and $S_e \subset E(H)$.
		Since $S_e$ is an equivalence class, every vertex $u$ in $V(G) \setminus V(S_e)$ is such that if $u$ is adjacent to a vertex in $V(S_e)$, then $u$ is adjacent to every vertex in $V(S_e)$.
		Thus the choice of $H$ is not unique since $S_e$ satisfies all of the necessary requirements.
		This is a contradiction.
				
		\emph{Case 2: $E(G) \setminus E(H)$ contains multiple equivalence classes of $\mathcal{C}_G$.}\\
		By Lemma~\ref{lem:noteveryec}, if any two equivalence classes, $S_e$ and $S_f$, are such that $V(S_e) = V(S_f)$, then $S_e = S_f$.
		Therefore, since $E(G) \setminus E(H)$ contains multiple equivalence classes, some equivalence class, $S_e \subset E(G) \setminus E(H)$, must be such that $V(S_e) \neq V(G)$.
		Let $S_d$ be an equivalence class such that $S_d \subset (E(G) \setminus E(H))$ and $V(S_d) \neq V(G)$.
		The graph $G[S_d]$ is connected, by Corollary~\ref{cor:connected}, and $V(S_d)$ has between $2$ and $n-1$ vertices.
		Also since every induced copy of $P_3$ either has both or neither of its edges in $S_d$, we have that for every vertex $u$ in $V(G) \setminus V(S_d)$, if $u$ is adjacent to a vertex in $V(S_d)$, then $u$ is adjacent to every vertex in $V(S_d)$.
		Thus the choice of $H$ is not unique, and this is a contradiction. 
		Therefore, $G$ is \uniquelycolourable.

	\end{proof}
	
	Using Theorem \ref{thm:UniqueHF1F2}, we provide the following example of an infinite family of graphs that admit a unique quasi-transitive $2$-edge-colouring.
	
	Let $P_k = w_0,w_1,w_2,\dots, w_{k-1}$.
	Construct $G$ from $P_k$ by adding vertices $u$ and $v$ and edges $uv, uw_0$ and $vw_0$.
	By observation we have $S_{uv} = \{uv\}$ and $S_e = \{uw_0, vw_0\} \cup \{w_i,w_{i+1} | 0 \leq i \leq k-1\}$ for all $e \neq uv$.
	Therefore $E(G)/\mathcal{C} = \{ S_{uv},  S_{uw_0}\}$.
	
	By Corollary \ref{cor:countingColourings}, it follows that $G$ is \uniquelycolourable.
	In the statement of Theorem \ref{thm:UniqueHF1F2}, the edge $uv$ plays the role of $H$.
		
	In this construction, each of the subgraphs induced by the two equivalence classes satisfy the criteria of Corollary \ref{cor:noColouring}.
	That is, the equivalence classes induce an partition of the graph into two subgraphs that admit only trivial quasi-transitive $2$-edge-colourings.
	Such a partition exists in general.
	
	\begin{theorem}\label{thm:ECnotcolourable}
		Let $G$ be a graph. If $S_e$ is an equivalence class of $\mathcal{C}_G$, then $G[S_e]$ is not \colourable. 
	\end{theorem}
	
	\begin{proof}
		Let $G$ be a graph, let $n$ be the number of vertices in $G$, and let $e$ be an edge in $E(G)$.
		Toward a contradiction, suppose that $G[S_e]$ is \colourable.
		By Theorem~\ref{thm:HF1F2}, there exists an induced proper subgraph $H$ of $G[S_e]$ such that
		\begin{itemize}
		\item $H$ is connected;
		\item $2 \leq |V(H)| \leq n-1$; and 
		\item for every vertex $u$ in $V(G) \setminus V(H)$, if $u$ is adjacent to a vertex in $V(H)$, then $u$ is adjacent to every vertex in $V(H)$.
		\end{itemize}	
		Since $S_e$ is an equivalence class of $\mathcal{C}_G$, for every vertex $u$ in $V(G) \setminus V(S_e)$, if $u$ is adjacent to a vertex in $V(S_e)$, then $u$ is adjacent to every vertex in $V(S_e)$.
		So every vertex in $V(G) \setminus V(H)$ that is adjacent to some vertex in $V(H)$ must be adjacent to every vertex in $V(H)$. 
		Thus for all $h$ in $E(H)$, the subset $S_h$ of $E(G)$ must equal $E(H)$.
		Therefore, $E(H)$ is an equivalence class of $\mathcal{C}_G$.
		This is a contradiction because $S_e$ is an equivalence class of $\mathcal{G}$ and $E(H)$ is a proper subset of $S_e$.
		Therefore, $G[S_e]$ is not \colourable. 
	\end{proof}
	
	We conclude this section with two more results that further aid in understanding the structure of graphs that contain multiple equivalence classes.
	
	\begin{theorem}\label{thm:pathsingleEC}
		Let $G$ be a connected graph and let $u$ and $v$ be in $V(G)$. Every shortest path between $x$ and $y$ in $G$ must only contain edges from a single equivalence class of $\mathcal{C}_G$.
	\end{theorem}
	
	\begin{proof}
		Let $G$ be a graph and let $x$ and $y$ be in $V(G)$. 
		Toward a contradiction, suppose there exists a shortest path $P=v_1 \dots v_k$ from $x$ to $y$, with $v_1=x$, $v_k=y$, and that some pair of incident edges in $P$, $v_{i-1}v_i$ and $v_{i}v_{i+1}$, are such that $S_{v_{i-1}v_i}$ is not equal to $S_{v_iv_{i+1}}$. By Corollary~\ref{cor:dichromaticargument}, the edge $v_{i-1}v_{i+1}$ must exist because otherwise an induced $P_3$ with edges from two different equivalence classes would result. The contradiction arises because the shortest path from $x$ to $y$, $P$, can be made shorter by replacing the edges $v_{i-1}v_i$ and $v_{i}v_{i+1}$ by the edge $v_{i-1}v_{i+1}$. Therefore, every shortest path between $x$ and $y$ in $G$ must only contain edges from a single equivalence class of $\mathcal{C}_G$.
	\end{proof}

	\begin{corollary}\label{cor:twopendants}
		Let $G$ be a connected graph. 
		In $G$, there is a maximum of one equivalence class of $\mathcal{C}_G$ that contains a vertex as an endpoint that is not found as an endpoint in any other equivalence class.
	\end{corollary}
	
	\begin{proof}
		Let $v$ be a vertex that is an endpoint in exactly one equivalence class of $\mathcal{C}_G$, say $v$ is in $V(S_e)$ for some edge $e$ in $E(G)$. 
		Toward a contradiction, suppose that $u$ is in $V(S_f)$ for some edge $f$ in $E(G)$, and that $u$ is an endpoint in exactly one equivalence class as well. 
		Let $P$ be the shortest path connecting $u$ to $v$. 
		By Theorem~\ref{thm:pathsingleEC}, $P$ only contains edges from a single equivalence class.
		This is a contradiction because $u$ and $v$ are not endpoints of edges in the same equivalence class.
		Therefore, there is a maximum of one equivalence class of $\mathcal{C}_G$ that contains a vertex as an endpoint that is not found as an endpoint in any other equivalence class.
	\end{proof}
	
	\section{Intersection of Equivalence Classes} \label{sec:3equivClasses}	 
	In Section~\ref{sec:SECHF1F2}, we characterized graphs with at least two equivalence classes (Theorem~\ref{thm:HF1F2}) and graphs with exactly two equivalence classes (Theorem~\ref{thm:UniqueHF1F2}). 
	In this section, we characterize graphs with exactly three equivalence classes. 
	We approach our study by considering the structure of the subgraph induced by a pair of equivalence classes.
	
	By Corollary~\ref{cor:subset}, when graphs have exactly two equivalence classes, $S_e$ and $S_f$, either $V(S_e) \subset V(S_f)$ or $V(S_f) \subset V(S_e)$. 
	However, when a graph has more than two equivalence classes, this subset property does not necessarily hold.
	Let $S_e$ and $S_f$ be distinct equivalence classes of $\mathcal{C}_G$. 
	It is not necessary for there to be any vertices in $V(S_e) \cap V(S_f)$. 
	However if $G$ is a connected graph with at least two equivalence classes, then there must exist some vertex $v$ and some pair of equivalence classes, $S_e$ and $S_f$, such that $v$ is in both $V(S_e)$ and $V(S_f)$. 
	For the purposes of this section, we will suppose, without loss of generality, the set $V(S_e) \cap V(S_f)$ is nonempty (see Figure~\ref{fig:int1}).
	
	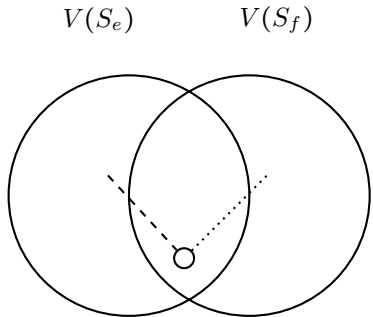
\begin{figure}[H]
	\centering
	
	\begin{tikzpicture}[-,-=stealth', auto,node distance=1.5cm,
		thick,scale=0.4, main node/.style={scale=0.8,circle,draw,font=\sffamily\Large\bfseries}]
		
		\node[draw=none,fill=none] (1) 					        {};
		\node[draw=none,fill=none] (100) [above = 1.5cm of 1]		{$V(S_e)$};
		\node[main node] (2)  [below right = 1cm and 0.9cm of 1]        {};
		\node[draw=none, fill=none] (3)  [above right = 1cm and 1cm of 2]        {}; 
		\node[draw=none,fill=none] (101) [above = 1.5cm of 3]		{$V(S_f)$};

		\draw[]
		(1,-1) circle (4cm)
		(5,-1) circle (4cm);
		
		\draw[dashed]
		(1) -- (2);
		
		\draw[dotted]
		(2) -- (3);
		
	\end{tikzpicture}  
	\caption{The vertex sets of two equivalence classes $S_e$ (represented by dashed edges) and $S_f$ (represented by dotted edges) shown with a nonempty intersection}
	\label{fig:int1}
\end{figure}

	Every edge with endpoints in both $V(S_e) \setminus V(S_f)$ and $V(S_f) \setminus V(S_e)$ must be in neither $S_e$ nor $S_f$. 
	There exist cases in which a single vertex is in the intersection of two equivalence classes and cases where the intersection contains two non-adjacent vertices (both cases arise in $K_4 \setminus \{e\}$, see Figure~\ref{fig:K4-e}).

		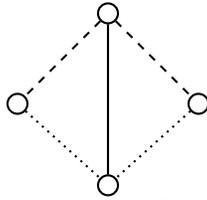
\begin{figure}[H]
	\centering

	\begin{tikzpicture}[-,-=stealth', auto,node distance=1.5cm,
		thick,scale=0.4, main node/.style={scale=0.8,circle,draw,font=\sffamily\Large\bfseries}]
		
		\node[main node] (1) 					        {};			
		\node[main node] (2)  [below left = 1cm and 1cm of 1]         {};
		\node[main node] (3)  [below right = 1cm and 1cm of 1]        {}; 
		\node[main node] (4)  [below = 2cm of 1]	       {};

		\draw[]
		(1) -- (4)
		;
		\draw[dashed]
		(1) -- (2)
		(1) -- (3)
		;
		\draw[dotted]
		(3) -- (4)
		(2) -- (4);
		
	\end{tikzpicture}

	\caption{$K_4 \setminus \{e\}$ shown with the three equivalence classes represented by full, dashed, and dotted lines.}
	\label{fig:K4-e}
\end{figure}
	
	For the purposes of this section, we will suppose that $V(S_e) \setminus V(S_f)$, $V(S_f) \setminus V(S_e)$, and $V(S_e) \cap V(S_f)$ are non-empty. We will prove a number of results regarding the types of graphs in which these three vertex sets are nonempty, culminating in a classification of graphs with exactly three equivalence classes (Theorem~\ref{thm:3Equiv}). We begin with a rather specific result, but one that will be useful in proving Theorem~\ref{thm:redmiddle}.
	
\begin{lemma}\label{lem:tinylemma}
	Let $G$ be a connected graph and let $e,f \in E(G)$ be such that $S_e$ and $S_f$ are equivalence classes of $\mathcal{C}_G$. 
	If
	\begin{itemize}
		\item $V(S_e) \setminus V(S_f)$ is nonempty;
		\item $u,v,y \in V(S_e) \cap V(S_f)$ and $x \in V(S_f) \setminus V(S_e)$;
		\item $uv, vx \in S_f$, $vy \notin S_f$, and $uy \in S_e$;
	\end{itemize}
	then $ux \in E(G)$.
\end{lemma}	

\begin{proof}
	Let $G$ be a connected graph and let $e,f \in (G)$ be such that $S_e$ and $S_f$ are equivalence classes of $\mathcal{C}_G$. 
	Suppose
	\begin{itemize}
		\item $V(S_e) \setminus V(S_f)$ is nonempty;
		\item $u,v,y \in V(S_e) \cap V(S_f)$ and $x \in V(S_f) \setminus V(S_e)$;
		\item $uv, vx \in S_f$, $vy \notin S_f$, and $uy \in S_e$.
	\end{itemize}
	
		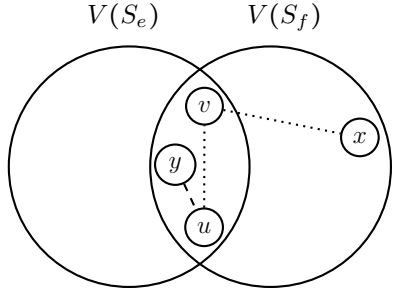
\begin{figure}[H]
	\centering

	\begin{tikzpicture}[-,-=stealth', auto,node distance=1.5cm,
		thick,scale=0.4, main node/.style={scale=0.8,circle,draw,font=\sffamily\Large\bfseries}]
		
		\node[draw=none, fill=none] (1) 					        {};			
		\node[main node] (2)  [below right = 0.05cm and 1.8cm of 1]        {$u$};
		\node[main node] (4)  [above left = 0.45cm and 0.001cm of 2]	       {$y$};			
		\node[main node] (5)  [above right = 0.4cm and 0.01cm of 4]        {$v$};
		\node[main node] (6)  [below right = 0.05cm and 1.7cm of 5]        {$x$};
		\node[draw=none,fill=none] (100) [above right = 2cm and 0.3cm of 1]		{$V(S_e)$};
		\node[draw=none,fill=none] (101) [right = 0.9cm of 100]		{$V(S_f)$};

		\draw[]
		(2.8,1.1) circle (4cm)
		(7.5,1.1) circle (4cm);
		
		\draw[]
		;
		\draw[dashed]
		(4) -- (2)
		;
		\draw[dotted]
		(2) -- (5)
		(5) -- (6);
		
	\end{tikzpicture}

	\caption{Theoretical intersection $V(S_e) \cap V(S_f)$ containing both endpoints of an edge in $S_f$ and both endpoints of an edge in $S_e$. Dotted edges are in $S_f$ and dashed edges are in $S_e$.}
	\label{fig:int4}
\end{figure}
	
	By Corollary~\ref{cor:dichromaticargument}, since $vy$ and $vx$ belong to different equivalence classes, the edge $xy$ must exist.
	Since $x \in V(S_f) \setminus V(S_e)$, the edge $xy$ is not in $S_e$.
	Therefore by Corollary~\ref{cor:dichromaticargument}, since $xy$ and $uy$ belong to different equivalence classes, the edge $ux$ must exist. 	
\end{proof}

\begin{theorem} \label{thm:redmiddle}
	Let $G$ be a connected graph and let $e,f \in E(G)$ be such that $S_e$ and $S_f$ are equivalence classes of $\mathcal{C}_G$. 
	If $V(S_e) \setminus V(S_f)$, $V(S_f) \setminus V(S_e)$, and $V(S_e) \cap V(S_f)$ are all nonempty, then $V(S_e) \cap V(S_f)$ does not contain any pair of vertices $u$ and $v$ such that $uv$ is in either $S_e$ or $S_f$. 
\end{theorem}

\begin{proof}
	Toward a contradiction, suppose that $u$ and $v$ are in $V(S_e) \cap V(S_f)$ and that $uv$ is in $S_f$. 
	In order for $S_f$ to not be a singleton set, $uv$ must exist in some induced copy of $P_3$ with another edge from $S_f$. 
	Without loss of generality, suppose $xv$ is in $S_f$ and that $uvx$ is an induced copy of $P_3$. 
	Every vertex in $V(S_e) \cap V(S_f)$ must be incident with an edge in $S_e$.
	
	\emph{Case 1: There does not exist a vertex in $V(S_e) \cap V(S_f)$ that is adjacent to $v$ with an edge in $S_f$ and also adjacent to a vertex in $V(S_e) \setminus V(S_f)$ with an edge in $S_e$.}\\
	
	So every edge in $S_e$ that is incident to $u$ must have both endpoints in $V(S_e) \cap V(S_f)$.
	Let $uy$ be such an edge.
	By Corollary~\ref{cor:dichromaticargument}, since $uv$ and $uy$ belong to different equivalence classes, the edge $vy$ must exist.
	By Lemma~\ref{lem:tinylemma}, if $vy \notin S_f$, then $ux \in E(G)$, contradicting $uvx$ being an induced copy of $P_3$. 
	So $vy \in S_f$.
	However, if every neighbour of $u$ via $S_e$ edges is a vertex in $V(S_e) \cap V(S_f)$ that is adjacent to $v$ via an $S_f$ edge, then there does not exist an $S_e$ edge with one endpoint in $V(S_e) \setminus V(S_f)$ and the other endpoint in $V(S_e) \cap V(S_f)$. 
	Therefore since $V(S_e) \setminus V(S_f)$ is nonempty, $G[S_e]$ is not connected.
	This contradicts Corollary~\ref{cor:connected}.
	
	\emph{Case 2: Some vertex in $V(S_e) \cap V(S_f)$ that is adjacent to $v$ with an edge in $S_f$ is also adjacent to a vertex in $V(S_e) \setminus V(S_f)$ with an edge in $S_e$.}\\
	
	Without loss of generality, suppose that $u$ is such a vertex.
	Let $y$ be a vertex in $V(S_e) \setminus V(S_f)$ such that $uy \in S_e$.
	By Corollary~\ref{cor:dichromaticargument}, since $uv$ and $uy$ belong to different equivalence classes, the edge $vy$ must exist.
	The edge $vy$ is not in $S_f$ since $y \in V(S_e) \setminus V(S_f)$.
	So by Corollary~\ref{cor:dichromaticargument}, since $vx$ and $vy$ belong to different equivalence classes, the edge $xy$ must exist.
	The edge $xy$ is not in $S_e$ since $x \in V(S_f) \setminus V(S_e)$.
	Finally by Corollary~\ref{cor:dichromaticargument}, since $uy$ and $xy$ belong to different equivalence classes, the edge $ux$ must exist.
	This contradicts $uvx$ being an induced copy of $P_3$.
	
	Therefore, there does not exist a pair of adjacent vertices $u,v \in V(S_e) \cap V(S_f)$ such that $uv$ is in either $S_e$ or $S_f$.

\end{proof}
	
	The following four results will build upon each other and be used to justify Theorem~\ref{thm:completetripartite}, and then finally yield our classification of those graphs that have exactly three equivalence classes (Theorem~\ref{thm:3Equiv}).
		
	\begin{lemma} \label{lem:allred}
		Let $G$ be a connected graph and let $S_e$ and $S_f$ be two equivalence classes of $\mathcal{C}_G$.
		If $V(S_e) \setminus V(S_f)$, $V(S_f) \setminus V(S_e)$, and $V(S_e) \cap V(S_f)$ are all nonempty, then every vertex in $V(S_e) \setminus V(S_f)$ is adjacent to every vertex in $V(S_f)$ and every vertex in $V(S_f) \setminus V(S_e)$ is adjacent to every vertex $V(S_e)$.		
	\end{lemma}
	
	\begin{proof}

		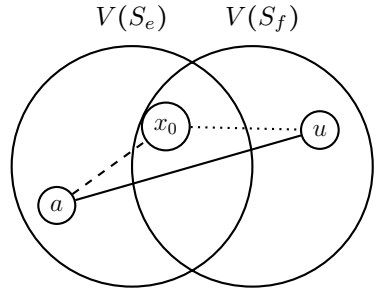
\begin{figure}[H]
	\centering
	
	\begin{tikzpicture}[-,-=stealth', auto,node distance=1.5cm,
		thick,scale=0.4, main node/.style={scale=0.8,circle,draw,font=\sffamily\Large\bfseries}]
		
		\node[main node] (1) 					        {$a$};			
		\node[draw=none, fill=none] (2)  [below right = 2cm and 5cm of 1]        {};
		\node[main node] (3)  [above left = 3cm and 1.5cm of 2]        {$u$}; 
		\node[draw=none, fill=none] (4)  [above = 2cm of 1]	       {};			
		\node[draw=none, fill=none] (5)  [above right = 2cm and 5cm of 4]        {};
		\node[main node] (7)  [above left = 3cm and 3.5cm of 2]         {$x_0$};
		\node[draw=none,fill=none] (100) [above right = 2cm and 0.2cm of 1]		{$V(S_e)$};
		\node[draw=none,fill=none] (101) [right = 0.5cm of 100]		{$V(S_f)$};

		\draw[]
		(2.5,1.3) circle (4cm)
		(6.5,1.3) circle (4cm);
		
		\draw[]
		(1) -- (3)
		;
		
		\draw[dashed]
		(1) -- (7)
		;
		
		\draw[dotted]
		(7) -- (3)
		;
		
	\end{tikzpicture}
	
	\caption{The vertex sets of two equivalence classes with nonempty intersection. Dotted edges are in $S_f$ and dashed edges are in $S_e$. The full edge is from a third equivalence class, $S_{au}$.}
	\label{fig:int2}
	
\end{figure}

		Let $x_0 \in V(S_e) \cap V(S_f)$, and $a \in V(S_e) \setminus V(S_f)$ such that $ax_0 \in S_e$. 
		Such a pair of vertices must exist because otherwise, with no edge with one endpoint in $V(S_e) \cap V(S_f)$ and one endpoint in $V(S_e) \setminus V(S_f)$, the graph $G[S_e]$ would be disconnected, contradicting Corollary~\ref{cor:connected}. 
		Since $x_0$ is in $V(S_e) \cap V(S_f)$, there exists a vertex $u$ such that $ux_0$ is in $S_f$.
		The vertex $u$ must be in $V(S_f) \setminus V(S_e)$ by Theorem~\ref{thm:redmiddle}. 
		By Corollary~\ref{cor:dichromaticargument}, since $ax_0$ and $ux_0$ are in different equivalence classes, the edge $au$ must exist.
		The vertex $a$ is not incident with any edges in $S_f$ and the vertex $u$ is not incident with any edges in $S_e$. 
		We will use induction to show that every vertex in $V(S_f)$ is adjacent to $a$ (and equivalently, every vertex in $V(S_e)$ is adjacent to $u$), inducting on the length of the shortest $S_f$-path connecting a vertex to $x_0$.
		By Theorem~\ref{thm:redmiddle}, every vertex that is adjacent to $x_0$ with an edge in $S_f$ must be in $V(S_f) \setminus V(S_e)$.
		Thus by Corollary~\ref{cor:dichromaticargument}, every vertex that is adjacent to $x_0$ with an edge in $S_f$ must also be adjacent to $a$ with an edge in neither $S_e$ nor $S_f$.
		Now suppose that every vertex $v$, such that the shortest $S_f$-path connecting $v$ to $x_0$ is length $k$, is adjacent to $a$. 
		If there does not exist some vertex $v$ such that the shortest $S_f$-path connecting $v$ to $x_0$ is length $k+1$, then $a$ is adjacent to every vertex in $V(S_f)$.
		Otherwise, let $P = x_0 \dots x_{k+1}$ be an $S_f$-$(k+1)$-path.
		Then by Corollary~\ref{cor:dichromaticargument}, since $a$ is adjacent to $x_k$ with an edge that is not in $S_f$, the edge $ax_{k+1}$ must exist.
		Therefore, by induction, the vertex $a$ is adjacent to every vertex in $S_f$ (and equivalently, every vertex in $V(S_e)$ is adjacent to $u$).
		
		We now prove that every other vertex in $V(S_e) \setminus V(S_f)$ is adjacent to every vertex in $V(S_f)$ (and equivalently, that every other vertex in $V(S_f) \setminus V(S_e)$ is adjacent to every vertex in $V(S_e)$).
		Let $z$ be a vertex in $V(S_e) \setminus V(S_f)$.
		By Corollary~\ref{cor:dichromaticargument}, since the vertex $z$ is adjacent to the vertex $u$, the vertex $z$ is also adjacent to every vertex which is adjacent to $u$ with an edge in $S_f$. 
		We will induct on the distance from $u$ along $S_f$-paths.
		Suppose that every vertex $y$, such that the shortest $S_f$-path connecting $y$ to $u$ is length $k$, is adjacent to $z$. 
		If there does not exist some vertex $y$ such that the shortest $S_f$-path connecting $y$ to $u$ is length $k+1$, then $z$ is adjacent to every vertex in $V(S_f)$.
		Otherwise, let $P = u y_1 \dots y_{k+1}$ be an $S_f$-$(k+1)$-path.
		Then by Corollary~\ref{cor:dichromaticargument}, since $z$ is adjacent to $y_k$ with an edge that is not in $S_f$, the edge $zy_{k+1}$ must exist.
		Therefore, by induction, the vertex $z$ is adjacent to every vertex in $S_f$.
		Thus every vertex in $V(S_e) \setminus V(S_f)$ is adjacent to every vertex in $V(S_f)$. 
		Equivalently, every vertex in $V(S_f) \setminus V(S_e)$ is adjacent to every vertex in $V(S_e)$.	
	\end{proof}
	
	\begin{corollary} \label{cor:allcross}
		Let $G$ be a connected graph, and let $S_e$ and $S_f$ be equivalence classes of $\mathcal{C}_G$. 
		If $V(S_e) \setminus V(S_f)$, $V(S_f) \setminus V(S_e)$, and $V(S_e) \cap V(S_f)$ are all nonempty, then every edge in $S_e$ and every edge in $S_f$ has an endpoint in $V(S_e) \cap V(S_f)$.
	\end{corollary}
	
	\begin{proof}
		Let $G$ be a connected graph, and let $S_e$ and $S_f$ be equivalence classes of $\mathcal{C}_G$. 
		Toward a contradiction, suppose that $f$ has both endpoints in $V(S_f) \setminus V(S_e)$ and that both $V(S_e) \cap V(S_f)$ and $V(S_e) \setminus V(S_f)$ are nonempty.
		By Lemma~\ref{lem:allred}, the vertices in $V(S_e) \cap V(S_f)$ are joined to the vertices in $V(S_f) \setminus V(S_e)$.
		So there does not exist an induced copy of $P_3$ containing an edge with both endpoints in $V(S_f) \setminus V(S_e)$ and an edge with one endpoint in $V(S_f) \cap V(S_e)$.
		So $S_f$ does not contain any edges with an endpoint in $V(S_e) \cap V(S_f)$.
		This is a contradiction though as $V(S_e) \cap V(S_f)$ is nonempty.
		Therefore, if $V(S_e) \cap V(S_f)$ is nonempty, then every edge in $S_f$ (and thus, $S_e$) has an endpoint in $V(S_e) \cap V(S_f)$.   
	\end{proof}

	\begin{lemma} \label{lem:notajoin}
		Let $G$ be a connected graph, and let $S_e$ and $S_f$ be equivalence classes of $\mathcal{C}_G$. 
		If $V(S_e) \setminus V(S_f)$, $V(S_f) \setminus V(S_e)$, and $V(S_e) \cap V(S_f)$ are all nonempty, then none of those three vertex sets induces a join.
	\end{lemma}
	
	\begin{proof}
		We will first prove that $V(S_e) \setminus V(S_f)$ (and thus, $V(S_f) \setminus V(S_e)$) does not induce a join.  
		By Corollary~\ref{cor:allcross}, every edge in $S_e$ has an endpoint in $V(S_e) \cap V(S_f)$. 
		Toward a contradiction, suppose that the subgraph induced by $V(S_e) \setminus V(S_f)$ is a join. 
		Call the two joined vertex sets $A$ and $B$, so that $A \cup B = V(S_e) \setminus V(S_f)$. 
		Suppose, without loss of generality, that $e$ has an endpoint in $A$. 
		We know that $G[S_e]$ is a connected graph, by Corollary~\ref{cor:connected}, and that $B$ is joined not only to $A$, but also to $A \cup (V(S_e) \cap V(S_f))$. 
		Hence there cannot exist an induced copy of $P_3$ with one edge having both endpoints in $A \cup (V(S_e) \cap V(S_f))$ and the other edge having one endpoint in $B$. Thus, no edge with an endpoint in $B$ is in $S_e$. 
		This is a contradiction. 
		Neither $V(S_e) \setminus V(S_f)$ nor $V(S_f) \setminus V(S_e)$ is a join. 
		
		Now we prove that $V(S_e) \cap V(S_f)$ does not induce a join using a similar argument.
		By Corollary~\ref{cor:allcross}, every edge in $S_e$ has an endpoint in $V(S_e) \cap V(S_f)$. 
		Toward a contradiction, suppose that the subgraph induced by $V(S_e) \cap V(S_f)$ is a join. 
		Call the two joined vertex sets $A ^\prime$ and $B ^\prime$, so that $A ^\prime \cup B ^\prime = V(S_e) \cap V(S_f)$. 
		Suppose, without loss of generality, that $e$ has an endpoint in $A ^\prime$. 
		We know that $G[S_e]$ is a connected graph, by Corollary~\ref{cor:connected}, and that $B ^\prime$ is joined to not only $A ^\prime$, but also to $A ^\prime \cup (V(S_e) \setminus V(S_f))$. 
		Hence there cannot exist an induced copy of $P_3$ with one edge having both endpoints in $A ^\prime \cup (V(S_e) \setminus V(S_f))$ and the other edge having one endpoint in $B ^\prime$. 
		Thus, no edge with an endpoint in $B ^\prime$ is in $S_e$. 
		This is a contradiction. 
		Therefore, $V(S_e) \cap V(S_f)$ is not a join. 
	\end{proof}
	
	\begin{lemma}\label{lem:allblack}
	Let $G$ be a connected graph and let $S_e$ and $S_f$ be equivalence classes of $\mathcal{C}_G$. 
	If $V(S_e) \setminus V(S_f)$, $V(S_f) \setminus V(S_e)$, and $V(S_e) \cap V(S_f)$ are all nonempty, then every edge with one endpoint in $V(S_e) \setminus V(S_f)$ and one endpoint in $V(S_f) \setminus V(S_e)$ belongs to the same equivalence class of $\mathcal{C}_G$, call it $S_d$.
\end{lemma}

\begin{proof}
	Let $G$ be a connected graph and let $S_e$ and $S_f$ be equivalence classes of $\mathcal{C}_G$. 
	Suppose that $V(S_e) \setminus V(S_f)$, $V(S_f) \setminus V(S_e)$, and $V(S_e) \cap V(S_f)$ are all nonempty.
	We know from Lemma~\ref{lem:allred} that $V(S_e) \setminus V(S_f)$ is joined to $V(S_f) \setminus V(S_e)$ and that none of these edges belong to $S_e$ or $S_f$. 
	Let $v \in V(S_e) \setminus V(S_f)$. 
	It will be sufficient to prove that every $vx$, where $x$ is a vertex in $V(S_f) \setminus V(S_e)$, belongs to the same equivalence class. 
	Suppose $vw_0$ is in some equivalence class $S_d$ for some $w_0$ in $V(S_f) \setminus V(S_e)$. 
	By Lemma~\ref{lem:notajoin}, we know that $V(S_f) \setminus V(S_e)$ is not a join.
	Therefore, there exists some vertex $w_1$ in $V(S_f) \setminus V(S_e)$ that is not adjacent to $w_0$.
	Since $w_0vw_1$ is an induced copy of $P_3$, both $vw_0$ and $vw_1$ belong to the same equivalence class.
	Thus, $vw_1$ is in $S_d$.
	By induction, suppose that $v$ is adjacent to $k$ different vertices in $V(S_f) \setminus V(S_e)$, $w_0, \dots, w_k$, with edges in $S_d$.
	Let $G_k$ be the subgraph induced by $\{w_0, \dots, w_k, v\}$.
	Since $V(S_f) \setminus V(S_e)$ is not a join, if there exists a vertex in $V(S_f) \setminus V(S_e)$ that is not in $G_k$, then there exists such a vertex that is not adjacent to every vertex in $G_k$.
	Let $w_{k+1}$ be such a vertex and let $w_k$ be a vertex in $G_k$ not adjacent to $w_{k+1}$.
	Since $w_k$ is not adjacent to $w_{k+1}$, the edges $vw_k$ and $vw_{k+1}$ belong to the same equivalence class, $S_d$.
	Therefore for all $x$ in $V(S_f) \setminus V(S_e)$, $vx$ belongs to $S_d$.
	Therefore, every edge with one endpoint in $V(S_e) \setminus V(S_f)$ and one endpoint in $V(S_f) \setminus V(S_e)$ belongs to the same equivalence class.
\end{proof}

	\begin{theorem}\label{thm:completetripartite}
		If $G$ is a connected graph, $\mathcal{C}_G$ has exactly three equivalence classes, $S_e$ and $S_f$ are distinct equivalence classes of $\mathcal{C}_G$, and $V(S_e) \setminus V(S_f)$, $V(S_f) \setminus V(S_e)$, and $V(S_e) \cap V(S_f)$ are all nonempty, then $G$ is complete tripartite.
	\end{theorem}

	\begin{proof}
	Let $G$ be a connected graph, $\mathcal{C}_G$ have exactly three equivalence classes, $S_e$ and $S_f$ be distinct equivalence classes of $\mathcal{C}_G$, and $V(S_e) \setminus V(S_f)$, $V(S_f) \setminus V(S_e)$, and $V(S_e) \cap V(S_f)$ all be nonempty.
	By Lemma~\ref{lem:allred}, if $a$ is a vertex in $V(S_e) \setminus V(S_f)$ and $b$ is a vertex in $V(S_f) \setminus V(S_e)$, then $ab$ is an edge in $E(G)$.
	By Lemma~\ref{lem:allblack}, every such edge belongs to the same equivalence class, call it $S_{ab}$.
	By Lemma~\ref{lem:allred}, every vertex in $V(S_e) \cap V(S_f)$ is adjacent to every vertex in $V(S_e) \setminus V(S_f)$ and every vertex in $V(S_f) \setminus V(S_e)$.
	Now all that remains to be shown is that there are no edges in $E(G)$ such that both endpoints are in $V(S_e) \setminus V(S_f)$, $V(S_f) \setminus V(S_e)$, or $V(S_e) \cap V(S_f)$.
	Toward a contradiction, suppose that an edge $z$ exists with both endpoints in $V(S_e) \setminus V(S_f)$.
	If $z$ does not exist in an induced copy of $P_3$, then $z$ is itself a fourth equivalence class and this would be a contradiction. 
	So $z$ exists in some induced copy of $P_3$.
	Both endpoints of $z$ are adjacent to every vertex in $V(S_f) \setminus V(S_e)$ and $V(S_e) \cap V(S_f)$. 
	Therefore, any induced copy of $P_3$ which contains $z$ must only contain edges in $V(S_e) \setminus V(S_f)$.
	Thus no edge with both endpoints in $V(S_e) \setminus V(S_f)$ is in $S_e$, $S_f$, or $S_{ab}$.
	This is a contradiction since $G$ has exactly three equivalence classes.
	Therefore, no edge exists with both endpoints in $V(S_e) \setminus V(S_f)$ (or equivalently, $V(S_f) \setminus V(S_e)$).
	
	Now toward a contradiction, suppose that an edge $y$ exists with both endpoints in $V(S_e) \cap V(S_f)$.
	If $y$ does not exist in an induced copy of $P_3$, then $y$ is itself a fourth equivalence class and this would be a contradiction. 
	So $y$ exists in some induced copy of $P_3$.
	Both endpoints of $y$ are adjacent to every vertex in $V(S_f) \setminus V(S_e)$ and $V(S_e) \setminus V(S_f)$. 
	Therefore, any induced copy of $P_3$ which contains $y$ must only contain edges in $V(S_e) \cap V(S_f)$.
	Thus no edge with both endpoints in $V(S_e) \cap V(S_f)$ is in $S_e$, $S_f$, or $S_{ab}$.
	This is a contradiction since $G$ has exactly three equivalence classes.
	Therefore, no edge exists with both endpoints in $V(S_e) \cap V(S_f)$.
	So $G$ is complete tripartite.
	\end{proof}

Graphs with exactly three equivalence classes are classified as follows.

\begin{theorem}\label{thm:3Equiv}
	Let $G$ be a connected graph such that $\mathcal{C}_G$ has exactly three equivalence classes, $S_d$, $S_e$, and $S_f$. 
	If $G$ is not complete tripartite, then there exists $e \in E(G)$ so that $V(S_e) = V(G)$.
\end{theorem}

\begin{proof}
	Let $G$ be a connected graph such that $\mathcal{C}_G$ has exactly three equivalence classes.
	Since $G$ is connected, every equivalence class $S_x$ must be such that $V(S_x) \cap V(S_y) \neq \emptyset$ for some other equivalence class $S_y$. 
	Suppose that for every pair of equivalence classes $S_x$ and $S_y$, either $V(S_x) \setminus V(S_y)$, $V(S_y) \setminus V(S_x)$, or $V(S_x) \cap V(S_y)$ is empty, because otherwise by Theorem~\ref{thm:completetripartite}, $G$ is complete tripartite. 
	This implies that for every pair of equivalence classes, $S_x$ and $S_y$, either the intersection of their vertex sets is empty or one vertex set is a subset of the other.
	Since $G$ is connected, some pair of equivalence classes, $S_e$ and $S_f$, are such that $V(S_f) \subseteq V(S_e)$.
	Let $S_d$ be the other equivalence class.
	Since $G$ is connected, either $V(S_d)$ is a subset of $V(S_e)$ or $V(S_e)$ is a subset of $V(S_d)$.
	In either case, there exists an equivalence class for which the vertex set contains the vertex sets of all other equivalence classes.
	This implies there exists some equivalence class such that the vertex set contains every vertex in $G$.
\end{proof}

We provide the following infinite family of graphs that are not complete tripartite and have exactly three equivalence classes.
Let $P_k = v_0,v_1,\dots, v_{k-1}$ and $P^\prime_k = v^\prime_0,v^\prime_1,\dots, v^\prime_{k-1}$ be disjoint paths.
Let $G$ be the graph formed from $P_k$ and $P^\prime_k$ by adding a universal vertex $v$.
We claim $G$ has exactly three equivalence classes.

For each $0 \leq i,j \leq k-1$, the edge $v_iv_j^\prime$ does not exist.
Therefore the edges incident with $v$ are in the same equivalence class because of the induced copies of $P_3$ that exist.

Since every vertex not in $P_k$ is adjacent to either every vertex in $P_k$ or no vertex in $P_k$, by Lemma \ref{lem:ECHF1F2}, the equivalence class of any edge in $P_k$ contains only edges of $P_k$.
Therefore $S_{v_iv_{i+1}} = E(P_k)$ for all $0 \leq i \leq k-2$.
A similar argument implies  $S_{v^\prime_iv^\prime_i+1} = E(P^\prime_k)$ for all $0 \leq i \leq k-2$.
Therefore $E(G)/\mathcal{C} = \{ S_{v_0v_1}, S_{v^\prime_0v^\prime_1}, S_{uv_0} \}$.

Notice that in the statement of Theorem \ref{thm:3Equiv}, $uv_0$ plays the role of $e$.

\section{Uniquely Quasi-Transitively Orientable Graphs}\label{sec:QTOrient}
The statement of Corollary \ref{cor:noForbidden} hints at a significant difference between the family of quasi-transitively orientable graphs and \colourable\ graphs.
Recall that the family of quasi-transitively orientable graphs is equal to the family of comparability graphs (Theorem~\ref{thm:ghouila}).
Such graphs admit a forbidden subgraph characterisation \cite{H94} and thus can be identified in polynomial time.
On the other hand, by Theorem \ref{thm:HF1F2} and Corollary \ref{cor:noForbidden}, no such forbidden subgraph characterisation of \colourable\ graphs exists.
However, one may verify that the conditions of Theorem \ref{thm:HF1F2} can be checked in polynomial time.

In this section we restrict our attention to the family of comparability graphs.
Using techniques similar to those in Section \ref{sec:SECHF1F2} we find that sets of the form $S_e$ arise as equivalence classes in a relation related to quasi-transitive orientations of comparability graphs.

Let $G$ be a comparability graph and consider $uv \in E(G)$.
As $G$ is a comparability graph, there exists at least one quasi-transitive orientation of $G$ for which the edge $uv$ is oriented to have its head at $v$.
Note that if $vw \in E(G)$ and $uw \notin E(G)$, then necessarily, the edge $vw$ is oriented with its head at $v$ whenever $uv$ is oriented with its head at $v$. 

A graph $G$ is \textit{uniquely quasi-transitively orientable} if there exist exactly two quasi-transitive orientations of $G$. Given one of these two orientations, the other can be created by reversing the direction of each arc.
A \textit{partial quasi-transitive orientation generated by $\overrightarrow {uv}$ in $G$} is a quasi-transitive orientation of a largest subgraph (most edges) of $G$ that is uniquely quasi-transitively orientable, in which $\overrightarrow {uv}$ is an arc.
In Figure~\ref{fig:QTExample1}, we see the graph $G$ from Figure~\ref{fig:QT2EC} with a quasi-transitive orientation of the edges and we also see a partial quasi-transitive orientation generated by an edge in $G$.

\begin{figure}[]
	\centering	
	\begin{tikzpicture}[-,-=stealth', auto,node distance=1.5cm,
		thick,scale=0.8, main node/.style={scale=0.8,circle,draw,font=\sffamily\Large\bfseries}]
		
		\node[main node] (1) 					    {$v_1$};			
		\node[main node] (2)  [right = 2cm of 1]        {$v_2$};
		\node[main node] (3)  [below right = 2cm and 1cm of 2]        {$v_3$};  
		\node[main node] (4)  [below left = 2cm and 1cm of 3]	      {$v_4$};				
		\node[main node] (5)  [left = 2cm of 4]        {$v_5$};
		\node[main node] (6)  [above left = 2cm and 1cm of 5]        {$v_6$};
		
		\node[main node] (7)  [right = 4cm of 2]					    {$v_1$};			
		\node[main node] (8)  [right = 2cm of 7]        {$v_2$};
		\node[main node] (9)  [below right = 2cm and 1cm of 8]        {$v_3$};  
		\node[main node] (10)  [below left = 2cm and 1cm of 9]	      {$v_4$};				
		\node[main node] (11)  [left = 2cm of 10]        {$v_5$};
		\node[main node] (12)  [above left = 2cm and 1cm of 11]        {$v_6$};	
		
		

		
		\draw[->]
		(6) edge (2)
		(5) edge (4)
		(1) edge (4)
		(3) edge (2)
		(3) edge (4)
		(6) edge (4)
		(1) edge (2)
		(5) edge (2)
		
		(7) edge (10)
		(12) edge (8)
		(9) edge (8)
		(9) edge (10)
		(12) edge (10)
		(7) edge (8);
		

		\draw [->]
		(1) edge (5)
		(3) edge (5)
		(1) edge (6)
		(6) edge (5);
		
		
	\end{tikzpicture}
	\caption{Quasi-transitive orientation $Q$ of a graph $G$ and a partial quasi-transitive orientation of $Q$ generated by an arc.}
	\label{fig:QTExample1}	
\end{figure}
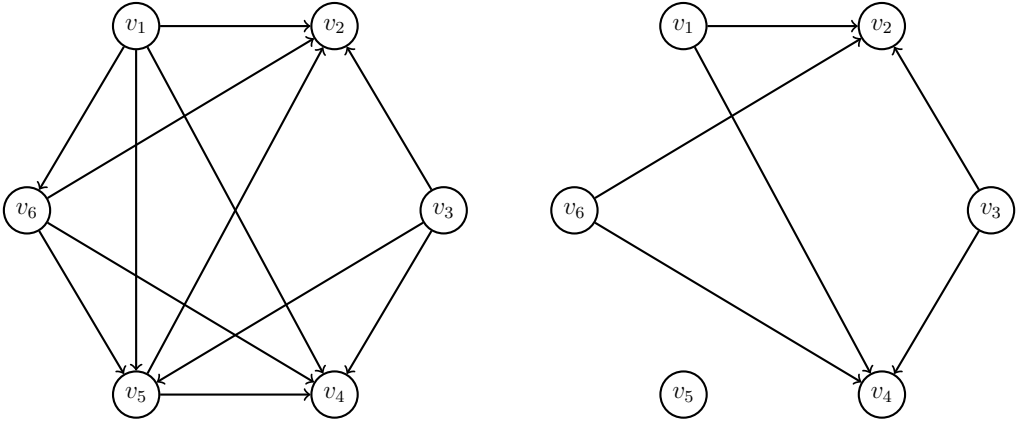

\begin{theorem}\label{thm:uniquePQTOGUV}
	For all comparability graphs $G$, if $uv \in E(G)$, then 
	
	\begin{itemize}
		\item there exists a unique partial quasi-transitive orientation generated by $\overrightarrow{uv}$ and
		\item the set of edges that are directed in the unique partial quasi-transitive orientation generated by $\overrightarrow{uv}$ is the equivalence class $S_{uv}$ of $\mathcal{C}_G$.
	\end{itemize}
\end{theorem}

\begin{proof}
	Suppose $G$ is a comparability graph and that $uv$ is chosen to be directed $\overrightarrow {uv}$ in a quasi-transitive orientation $X$ of $G$. 
	Every induced $P_3$ of the form $uvw$ is such that $\overrightarrow{wv}$ is an arc in $X$. 
	Also for every $P_3$ that contains either $uv$ or $vw$, the direction of the other edge is the same in all partial quasi-transitive orientations generated by $\overrightarrow{uv}$. 
	Let $S$ be a minimal subset of $E(G)$ which contains $uv$ and is such that every induced $P_3$ either has both or neither of its edges in $S$.
	By Lemma~\ref{lem:fewestunique}, $S$ is the only smallest subset of $E(G)$ which contains $uv$ and is such that every induced $P_3$ either has both or neither of its edges in $S$.
	No partial quasi-transitive orientation generated by $\overrightarrow{uv}$ can contain any edges not contained in $S$ because there would exist a partition of the edge set into two sets such that no induced copy of $P_3$ includes an edge from both sets.
	Thus $S$ is unique, and by Theorem~\ref{thm:equivalencerelation2}, the set $S$ is equal to the equivalence class $S_{uv}$ of $\mathcal{C}_G$.
\end{proof}	

We will denote the unique partial quasi-transitive orientation generated by $\overrightarrow{uv}$ in $G$ by $\Gamma_{\overrightarrow{uv}}$.
Let $\mathcal{O}_G$ be the relation on $E(G)$ so that $uv \sim xy$ when any of the following are true:
\begin{itemize}
	\item $\Gamma_{\overrightarrow{uv}} = \Gamma_{\overrightarrow{xy}}$
	\item $\Gamma_{\overrightarrow{uv}} = \Gamma_{\overrightarrow{yx}}$
	\item $\Gamma_{\overrightarrow{vu}} = \Gamma_{\overrightarrow{xy}}$
	\item $\Gamma_{\overrightarrow{vu}} = \Gamma_{\overrightarrow{yx}}$.
\end{itemize} 
Given a graph $G$, the relation $\mathcal{O}_G$ is an equivalence relation because both $\mathcal{O}_G$ and $\mathcal{C}_G$ partition the edges of $G$ in the same manner.

\begin{corollary}
	Let $G$ be a comparability graph. If $S_e$ is an equivalence class of $\mathcal{O}_G$, then $G[S_e]$ is uniquely quasi-transitively orientable.
\end{corollary}

\begin{corollary}
	For all integers $k \geq 1$, there exists a graph $G$ with $k$ equivalence classes under the relation $\mathcal{O}_G$.	
\end{corollary}

\begin{proof}
	By Theorem~\ref{thm:threshold} and Theorem~\ref{thm:uniquePQTOGUV}, for all integers $k \geq 1$, there exists a graph $G$ with $k$ equivalence classes under the relation $\mathcal{C}_G$ and both $\mathcal{C}_G$ and $\mathcal{O}_G$ partition the edges of $G$ in the same manner. 
	Therefore, for all integers $k \geq 1$, there exists a graph $G$ with $k$ equivalence classes under the relation $\mathcal{O}_G$.
\end{proof}

\begin{theorem}\label{thm:countingOrientations}
	If $G$ is a comparability graph so that $\mathcal{C}_G$ has $k$ equivalence classes, then there are $2^{k}$ quasi-transitive orientations of $G$.
\end{theorem}

\begin{proof}
	For all graphs $G$, the number of equivalence classes of $\mathcal{C}_G$ is equal to the number of equivalence classes of $\mathcal{O}_G$.
	For every arc $\overrightarrow{uv}$, there exists a unique partial quasi-transitive orientation generated by $\overrightarrow{uv}$ by Theorem~\ref{thm:uniquePQTOGUV}.
	Since in any graph orientation, every edge must be assigned one of two directions, for each equivalence class in $G$, there exist two possible orientations of those edges that can exist in a given quasi-transitive orientation of $G$.
	Therefore, there exist $2^k$ quasi-transitive orientations of $G$, where $k$ is the number of equivalence classes.

\end{proof}

\begin{corollary}
	A comparability graph $G$ is uniquely quasi-transitively orientable if and only if it is not \colourable.
\end{corollary}	

By the definitions of $S_{uv}$ and $[\overrightarrow{uv}]_{\mathcal{O}}$, Theorem \ref{thm:uniquePQTOGUV} tells us $S_{uv}$ is the largest set of edges such that for every quasi-transitive orientation $X$ of $G$ and every edge $ab \in S_{uv}$, if $\overrightarrow{ab}$ and $\overrightarrow{uv}$ are arcs in $X$, then $\overrightarrow{ab}$ is an arc in every quasi-transitive orientation of $G$ in which $\overrightarrow{uv}$ is an arc.
In Figure~\ref{fig:QTExample1}, the image on the right shows the set of oriented edges as exactly those of the set $S_{v_1v_2}$.

By Theorem \ref{thm:equivalencerelation2} and Theorem \ref{thm:uniquePQTOGUV}, we have $|E(G)/\mathcal{C}| = | E(G)/\mathcal{O}|$.
So we have the following result.
\begin{theorem}
	A comparability graph $G$ is uniquely quasi-transitively orientable if and only if $G$ admits only the trivial quasi-transitive $2$-edge-colouring.
\end{theorem}

 \section{Conclusions and Future Work}

Though the respective literature yields significant examples of commonalities in the study of oriented graphs and $2$-edge-coloured graphs, our work here highlights a fundamental difference between these two relational structures.
As with many concepts under study in oriented and $2$-edge-coloured graphs, the difference between the definitions of quasi-transitivity for oriented graphs and for $2$-edge-coloured graphs is only slight.
However in this case, the resulting classification is markedly different.

Such a phenomenon also occurs with the study of chromatic polynomials of oriented and $2$-edge-coloured graphs.
In \cite{C19}, Cox and Duffy fully classified those oriented graphs whose oriented chromatic polynomial was identical to the chromatic polynomials of the underlying simple graph as quasi-transitive orientations of co-interval graphs, notably a class of graphs for which the chromatic polynomial can be computed in polynomial time \cite{G77}.
In the sequel for $2$-edge-coloured graphs, Beaton, Cox, Duffy and Zolkavich \cite{B20} obtain only a partial classification of $2$-edge-coloured graphs whose $2$-edge-coloured chromatic polynomial is identical to the chromatic polynomial of the underlying simple graph.
Here, work was stymied by a paucity of research results on quasi-transitive $2$-edge-coloured graphs.
However, their preliminary results suggest that those $2$-edge-coloured graphs for which the $2$-edge-coloured chromatic polynomial is identical to the chromatic polynomial of the underlying simple graph will comprise a family of graphs for which computing the chromatic polynomial is NP-hard.
We expect that our results, particularly Theorem \ref{thm:HF1F2}, will provide the insight needed to complete this classification.

\bibliographystyle{abbrv}
\bibliography{references}

\end{document}